\documentclass[11pt]{amsart}  
\usepackage{amsmath,amsfonts,amssymb,mathrsfs}
\usepackage[english]{babel}
\usepackage[utf8]{inputenc}
\usepackage[T1]{fontenc}
\usepackage{pgf,tikz,pgfplots}
\pgfplotsset{compat=1.15}
\usepackage{mathrsfs}
\usepackage{bbold}
\usepackage{enumitem}
\usetikzlibrary{arrows}
\newtheorem{theorem}{Theorem}[section]
\newtheorem{thm}[theorem]{Theorem}
\newtheorem{lemma}[theorem]{Lemma}

\newtheorem{prop}[theorem]{Proposition}

\newtheorem{remark}{Remark}[section]

\def\R{\mathbb R}

\def\N{\mathbb N}
\def\Z{\mathbb Z}

\numberwithin{equation}{section}

\usepackage{hyperref}

\DeclareMathOperator{\dive}{div}

\usepackage{tikz}
\usetikzlibrary{decorations}
\usetikzlibrary{decorations.pathmorphing}
\usetikzlibrary{decorations.pathreplacing}
\usetikzlibrary{decorations.shapes}
\usetikzlibrary{decorations.text}
\usetikzlibrary{decorations.markings}
\usetikzlibrary{decorations.fractals}
\usetikzlibrary{decorations.footprints}
\tikzset{every picture/.style={execute at begin picture={
   \shorthandoff{:;!?};}
}}
\begin{document}
\title[Euler-Navier-Stokes system]{The pressureless Euler-Navier-Stokes system}

\author[Valentin Lemarié]{Valentin Lemarié}
 {\begin{center}
\begin{abstract} 
In this paper, we study the well-posedness of the pressureless Euler-Navier-Stokes system in $\R^d$ (with $d\geq 2$) in the critical regularity setting for a density close to $0$. We prove a global existence result for small data for this system, and then give optimal time decay estimates.
 \end{abstract}\end{center}}
\maketitle

\section{Introduction}
In this article, we focus on the pressureless Euler-Navier-Stokes system:
\begin{eqnarray}\label{Euler-Navier-Stokes1}
\left\{\begin{array}{l}
      \partial_t \rho+\dive(\rho w)=0, \\
     \partial_t (\rho w)+\dive(\rho w \otimes w)=-\rho(w-u), \\
     \partial_t u+\dive(u \otimes u)+\nabla P=\Delta u+\rho(w-u), \\
     \dive u=0,
\end{array} \right.
\end{eqnarray} with the initial data $(\rho,w,u)_{|t=0}=(\rho_0,w_0,u_0)$ around the constant state $(\overline{\rho}, \overline{w}, \overline{u})=(0,0,0)$.

The unknown functions $\rho=\rho(t,x)$ and $w=w(t,x)$, depending on time $t\geq 0$ and position $x\in\R^d$ with $d\geq 2$, respectively represent the non-negative density and velocity of the fluid flow satisfying the pressureless Euler equation; $u=u(t,x)$ and $P=P(t,x)$ are the velocity and pressure for the incompressible fluid satisfying the Navier-Stokes equation.

The Euler and Navier-Stokes equations play a central role in modelling fluid flows, in the ideal and viscous regimes respectively. While each of these dynamics is well studied independently, physical situations can lead to hybrid models, combining a viscous (incompressible) part and another of hyperbolic type, where pressure is neglected. This is the context of the pressureless Euler-Navier-Stokes system, which couples the compressible - pressureless - Euler equation to the incompressible Navier-Stokes equation. This system, although little studied in its own right, is of interest both physically and mathematically.

From a physical point of view, this type of coupling appears naturally in the modelling of dilute suspensions, when the particulate phase is described by simplified fluid dynamics (without pressure) and its interaction with the viscous carrier fluid remains significant. The system thus obtained retains the effects of transport, inertia and inter-phase friction, but ignores the effects of pressure in the particulate phase, which corresponds to a limit of low internal interaction in this phase. This type of simplification is motivated in particular by the study of the asymptotic regimes of the Vlasov-Navier-Stokes system, the mathematical study of which has attracted a great deal of attention over the last 20 years, due to its numerous applications in various fields: medicine \cite{Baranger}, diesel engine combustion \cite{O Rourke}, etc... By including a parameter $\varepsilon$ tending to $0$ in this system, numerous hydrodynamic limits resulting have been studied (see, for example, the article \cite{Michel et Daniel} by D. Han-Kwan and D. Michel or \cite{LucasToulouse} by L. Ertzbischoff).  However, the mathematical link, in the sense of the convergence of solutions of the Vlasov-Navier-Stokes system to those of the pressureless Euler-Navier-Stokes system, is still an open problem.

 The pressureless Euler-Navier-Stokes system that we propose to study here is a simplified but representative model, capturing both the fluid/particle coupling mechanisms and a rich mathematical structure.

From a mathematical point of view, global existence of solutions was established by Y.-P. Choi, J. Jung, and J. Kim in \cite{Choi2, Choi}. In the latter work, the authors consider small initial data in Sobolev spaces of non-critical regularity, together with an additional $L^1$ condition on $\rho_0$ and $u_0$. More recently, the pressureless Euler--Navier--Stokes system was also investigated in \cite{Danchin2} through a completely different approach based on the conservation structure of the equations. In particular, the authors establish global well-posedness in dimension three without any smallness assumption on the density, which is a remarkable result.

In contrast, we investigate the problem in a critical regularity framework within suitable Besov spaces, which in particular ensures that the density remains bounded in $L^\infty$. This property is crucial for preserving the strict positivity of the density, a necessary ingredient for rigorously justifying the equivalence between \eqref{Euler-Navier-Stokes1} and its non-conservative formulation \eqref{Euler-Navier-Stokes2}. Moreover, our functional framework is closely related to the multiphase formulation of kinetic-fluid models and therefore appears particularly well suited for future investigations of the hydrodynamic limit from the Vlasov--Navier--Stokes system towards the pressureless Euler--Navier--Stokes system. Finally, when assuming an additional $L^1$ condition on the velocities, which allows us to control the $L^\infty$ norm of the density, we obtain optimal decay estimates. Under the same regularity assumptions as in \cite{Choi}, we also recover the decay rates established in their work.

\section{Main results and sketch of the proof}
In this section, we present and motivate the functional spaces used.

Then, in a second step, we state the results and outline the proof.

\subsection{Functional spaces}~\\
Before setting out the main results of this article, we briefly explain the different notations and definitions used throughout the document. 

We will refer to $C>0$ a constant independent of $\varepsilon$ and time and $f\lesssim g$ will mean $f\leq C g$. For all Banach space $X$ and all functions $f,g\in X$, we set up $\|(f,g)\|_X\mathrel{\mathop:=}\|f\|_X+\|g\|_X$. 
\\
We denote by $L^2(\R_+;X)$ the set of measurable functions $f:[0,+\infty[\rightarrow X$ such that $t\mapsto \|f(t)\|_{X}$ is in $L^2(\R_+)$ and let us write $\|\cdot \|_{L^2(X)}\mathrel{\mathop:}=\|\cdot \|_{L^2(\R_+;X)}$.

In this article, we use a classical decomposition in Fourier space, called Littlewood-Paley's homogeneous dyadic decomposition.

Let $\varphi$ a non-negative regular function on $\R^d$ with support in the annulus \begin{equation}\label{anneau A}\mathcal{A}\mathrel{\mathop:}=\{\xi\in\R^d,\:  3/4\leq|\xi|\leq 8/3\}\end{equation} and satisfying $$\sum_{j\in\Z}\varphi(2^{-j}\xi)=1,\qquad \xi\not=0.$$ 

For all $j\in\Z$, dyadic homogeneous blocks $\dot \Delta_j$ and the low-frequency truncation operator $\dot S_j$ are defined by 
$$\dot \Delta_j \mathrel{\mathop:}=\mathcal{F}^{-1}(\varphi(2^{-j}\cdot)\mathcal{F}u), \quad \dot S_j u\mathrel{\mathop:}=\sum_{j'\leq j-1}\dot\Delta_{j'} \ u,$$ 
where $\mathcal{F}$ and $\mathcal{F}^{-1}$ denote the Fourier transform and its inverse respectively. 

By construction, $\dot\Delta_j$ is a localization operator around the frequency of magnitude $2^j$.

From now on, we will use the following shorter notation : $$ u_j:=\dot\Delta_j u.$$

Let $\mathcal{S}_h'$ the set of tempered distributions $u$ on $\R^d$ such that $$\displaystyle \underset{j\to -\infty}{\lim}\|\dot S_j u\|_{L^\infty}=0.$$ we then have : $$u=\sum_{j\in \Z} u_j \ \in \mathcal{S}', \quad \dot S_j u=\sum_{j'\leq j-1} u_{j'}, \ \forall u\in \mathcal{S}_h'.$$

By Bernstein's Lemma (Lemma 2.1 in \cite{BCD}), we have by definition of the annulus \eqref{anneau A} and dyadic blocks $\dot\Delta_j$,the existence of a constant $c_B$ satisfying the following inequality for all $u\in\mathcal{S}'(\R^d)$ : 
\begin{equation}\label{constante de Bernstein}
   c_B^{-\frac{k+1}{2}}2^{jk}\|\dot\Delta_j u\|_{L^2} \leq \|D^k \dot\Delta_j u\|_{L^2}\leq c_B^{\frac{k+1}{2}}2^{jk}\|\dot\Delta_j u\|_{L^2}.
\end{equation}

In terms of these dyadic blocks, the homogeneous Besov spaces $\dot B_{2,r}^s$ for all $s\in\R$ are defined by : $$\dot B_{2,1}^s\mathrel{\mathop:}=\left\{u\in \mathcal{S}_h' \middle| \|u\|_{\dot B_{2,r}^s}\mathrel{\mathop:}=\|\{2^{js}\|u_j\|_{L^2}\}_{j\in\Z}\|_{l^r}<\infty\right\}.$$

As we will need to restrict our Besov norms to specific regions of the low and high frequencies, we introduce the following notations : \begin{center}
$\displaystyle\|u\|_{\dot B_{2,1}^s}^h\mathrel{\mathop:}=\sum_{j\geq 0}2^{js}\|u_j\|_{L^2}$, $\displaystyle \ \|u\|_{\dot B_{2,1}^s}^l\mathrel{\mathop:}=\sum_{j\leq -1}2^{js}\|u_j\|_{L^2}$.
\end{center}

For comparison with Sobolev spaces $\dot H^s$ for $s\in\R$, we also have the following continuous injections:
$$\dot B_{2,1}^s\hookrightarrow\dot H^s\hookrightarrow \dot B_{2,\infty}^s$$

The appendix contains several results about Besov spaces: the reader can refer to Chapter 2 of \cite{BCD} for more information on this subject.

\subsection{Main results and organization of the article}~\\
In the present paper,  the pressureless Euler–Navier-Stokes system is considered in Besov spaces $\dot B^s_{2,1}.$  We work within a "critical" regularity framework that is motivated by previous studies of the equations of the system taken separately.
To be more precise, when decoupled, this system consists of the Euler equations for $w$ and the Navier–Stokes equations for $u$, for which the critical regularities are known, namely $\frac{d}{2}+1$ for Euler and $\frac{d}{2}-1$ for Navier–Stokes. The equation on $\rho$ suggests that there should be a difference of one derivative between $w$ and $\rho$, which naturally leads to choosing the index $\frac{d}{2}$ for $\rho$, with which we can use critical injection $\dot B_{2,1}^{\frac{d}{2}}(\mathbb{R}^d) \hookrightarrow L^\infty(\mathbb{R}^d)$.

As a first step, we prove in these spaces a global existence and uniqueness result for the following system:

\begin{eqnarray}\label{Euler-Navier-Stokes2}
    \left\{\begin{array}{l}
      \partial_t \rho+\dive(\rho w)=0, \\
     \partial_t w + (w\cdot \nabla)w+w-u=0, \\
     \partial_t u+(u\cdot \nabla) u+\nabla P=\Delta u+\rho(w-u), \\
     \dive u=0.
\end{array} \right.
\end{eqnarray}

\begin{thm}\label{théorème existence et unicité}
    There exists a non-negative constant $\alpha>0$ such that for all initial data $Z_0\mathrel{\mathop:}=(\rho_0,w_0,u_0)\in \dot B_{2,1}^{\frac{d}{2}}\times \left(\dot B_{2,1}^{\frac{d}{2}-1}\cap \dot B_{2,1}^{\frac{d}{2}+1}\right)\times \dot B_{2,1}^{\frac{d}{2}-1}$ satisfying : \begin{eqnarray}\label{condition initiale}\mathcal{Z}_0\mathrel{\mathop:}=\|\rho_0\|_{\dot B_{2,1}^{\frac{d}{2}}}+\|w_0\|_{\dot B_{2,1}^{\frac{d}{2}-1}\cap \dot B_{2,1}^{\frac{d}{2}+1}}+\|u_0\|_{\dot B_{2,1}^{\frac{d}{2}-1}}\leq \alpha,\end{eqnarray}
    the system \eqref{Euler-Navier-Stokes2} with the initial data $Z_0$ admits a unique global-in-time solution $(\rho,w,u,P)$ in the set
    \begin{multline}\label{espace fonctionnel E} E\mathrel{\mathop:}= \bigg\{ (\rho,w,u,P) \ \bigg| \ \rho\in \mathcal{C}_b(\R^+;\dot B_{2,1}^{\frac{d}{2}}), \ u\in \mathcal{C}_b(\R^+;\dot B_{2,1}^{\frac{d}{2}-1})\cap L^1(\R^+,\dot B_{2,1}^{\frac{d}{2}+1}), \\ w\in \mathcal{C}_b(\R^+; \dot B_{2,1}^{\frac{d}{2}-1}\cap \dot B_{2,1}^{\frac{d}{2}+1})\cap L^1(\R^+;\dot B_{2,1}^{\frac{d}{2}+1}), \nabla P\in L^1(\R^+,\dot B_{2,1}^{\frac{d}{2}-1}) \bigg\}.\end{multline}
Moreover, we have the following inequality for any $t\in\R^+$ : \begin{equation}\label{estimée théorème système 1}\mathcal{Z}(t)\leq C \mathcal{Z}_0 \end{equation} where $$\displaylines{\mathcal{Z}(t)\mathrel{\mathop:}= \|\rho\|_{L_t^\infty(\dot B_{2,1}^{\frac{d}{2}})}+\|w\|_{L_t^\infty(\dot B_{2,1}^{\frac{d}{2}-1}\cap \dot B_{2,1}^{\frac{d}{2}+1})}+\|u\|_{L_t^\infty(\dot B_{2,1}^{\frac{d}{2}-1})}+\|(w,u)\|_{L_t^1(\dot B_{2,1}^{\frac{d}{2}+1})} \hfill\cr\hfill +\|w-u\|_{L_t^\infty(\dot B_{2,1}^{\frac{d}{2}-1})\cap L_t^1(\dot B_{2,1}^{\frac{d}{2}-1})}+\|\nabla P\|_{L_t^1(\dot B_{2,1}^{\frac{d}{2}-1})}.}$$
\end{thm}
\begin{remark}
 Preserving the strict positivity of the density is important to show the equivalence between the systems \eqref{Euler-Navier-Stokes1} and \eqref{Euler-Navier-Stokes2}. 

    It comes from the characteristics method associated with the fact that $$\rho\in \mathcal{C}_b(\R^+; \dot B_{2,1}^{\frac{d}{2}})\hookrightarrow \mathcal{C}_b(\R^+; \mathcal{C}_0) \quad \text{and} \quad w\in L^1(\R^+; \dot B_{2,1}^{\frac{d}{2}+1})\hookrightarrow L^1(\R^+; C^1).$$ Indeed, we have: 
    $$\rho(t,X(t,x))\exp\left(\int_0^t \dive w(s,X(s,x))ds\right)=\rho_0(x),$$ where the flow $X$ is defined by $$\left\{\begin{array}{l} \displaystyle \frac{d}{dt}X(t;x)=w(t,X(t,x)), \\ X(t=0,x)=x\in\R^d.  \end{array}\right.$$
\end{remark}

If we assume an additional condition on the low frequencies of the initial velocities, we obtain, in the critical framework, an optimal time decay estimate:
\begin{theorem}\label{estimées de décroissance}
    Under the same assumptions as in Theorem \ref{théorème existence et unicité}, if the initial data $(\rho_0,w_0,u_0)$ also satisfies the hypothesis \begin{equation}\label{condition L1}w_0,u_0 \in\dot B_{2,\infty}^{-\frac{d}{2}}(\R^d),\end{equation} then $w$ and $u$ satisfy: \begin{equation}\label{taux de décroissance L2}\|(w,u)(t)\|_{L^2}\leq C (1+c_0 t)^{-\frac{d}{4}}.\end{equation} 
    For $d\geq 3$, we have in addition:
 \begin{equation}\label{taux de décroissance L2 mode amorti}\|(w-u)(t)\|_{L^2}\leq  C (1+c_0 t)^{-\frac{d(d+1)}{4(d-1)}}.\end{equation}
    In the case of $d=2$, we have in addition: 
    \begin{equation}\label{estimée décroissance u-v 2d}
        \|(w-u)(t)\|_{L^2(\R^2)}\leq C(1+c_0 t)^{-1}.
    \end{equation}
\end{theorem}

\begin{remark}~
    \begin{enumerate}
        \item Since we have the following continuous injection : $$L^1(\R^d) \hookrightarrow \dot B_{2,\infty}^{-\frac{d}{2}}(\R^d),$$ the initial data $w_0,u_0$ can be taken in $L^1$.
        \item The "damped mode" $w-u$ has a better decay than $u$ and $w$ taken separately: the frequency analysis below shows that this difference behaves better.
    \end{enumerate}
\end{remark}

Assuming more regularity on the initial data, we can improve the previous decay estimates:
\begin{theorem}\label{estimées de décroissance2}
Assume $d\geq 3$ and let the initial data $(\rho_0,w_0,u_0)$ satisfy the assumptions of Theorem \ref{estimées de décroissance}. If, in addition, for $k\geq 1$, \begin{equation}\label{condition régularité supplémentaire}\rho_0\in \dot B_{2,1}^{\frac{d}{2}+k}, \quad w_0\in\dot B_{2,1}^{\frac{d}{2}+k+1}  , \quad u_0\in\dot B_{2,1}^{\frac{d}{2}+k-1},\end{equation} then there exists a unique solution $(\rho,w,u)$, which belongs to $E$ defined in~\eqref{espace fonctionnel E}, where $w$ and $u$ also belong respectively to 
$$\mathcal{C}_b(\R^+;\dot B_{2,1}^{\frac{d}{2}+k+1})\cap L^1(\R^+; \dot B_{2,1}^{\frac{d}{2}+k+1}) \;  \text{and} \; \mathcal{C}_b(\R^+;\dot B_{2,1}^{\frac{d}{2}+k-1})\cap L^1(\R^+; \dot B_{2,1}^{\frac{d}{2}+k+1}).$$

Moreover, $w$ and $u$ satisfy the following inequality (for all $p\in [2,\infty]$): \begin{equation}\label{taux de décroissance L2 k}\|\nabla^{k-1}(w,u)\|_{L^p}\leq C (1+c_0 t)^{-\frac{d}{2}(1-\frac{1}{p})-\frac{k-1}{2}}.\end{equation} 
\end{theorem}

\subsection{Sketch of the proofs}~\\
To begin with, consider the Leray projector defined as follows :
$$\mathbb{P}\mathrel{\mathop:}=Id+\nabla(-\Delta)^{-1}\dive,$$
which is the orthogonal projector on fields with zero divergence.

Then, the study of \eqref{Euler-Navier-Stokes2} is equivalent to that of the following system:
\begin{eqnarray}\label{Euler-Navier-Stokes3}
    \left\{\begin{array}{l}
      \partial_t \rho+\dive(\rho w)=0, \\
     \partial_t w + (w\cdot \nabla)w+w-u=0, \\
     \partial_t u+\mathbb{P}(u\cdot \nabla) u=\Delta u+\mathbb{P}\left(\rho(w-u)\right).
\end{array} \right.
\end{eqnarray}
 $\nabla P$ may be computed from the relation \begin{eqnarray}\label{expression de la pression P} \nabla P=\nabla(-\Delta)^{-1}\dive\left((u\cdot \nabla)u\right)-\nabla(-\Delta)^{-1}\dive(\rho(w-u)),\end{eqnarray}
and $u$ satisfies $\dive u=0$.

Consequently, in order to prove Theorem \ref{théorème existence et unicité}, it suffices to prove the following theorem pertaining to the well-posedness of the system \eqref{Euler-Navier-Stokes3} : 
\begin{theorem}\label{théorème existence et unicité2}
    There exists a non-negative constant $\alpha>0$ such that for all initial data $Z_0\mathrel{\mathop:}=(\rho_0,w_0,u_0)\in \dot B_{2,1}^{\frac{d}{2}}\times \left(\dot B_{2,1}^{\frac{d}{2}-1}\cap \dot B_{2,1}^{\frac{d}{2}+1}\right)\times \dot B_{2,1}^{\frac{d}{2}-1}$ satisfying the condition \eqref{condition initiale},
    the system \eqref{Euler-Navier-Stokes3} with the initial data $Z_0$ admits a unique global-in-time solution $(\rho,w,u)$ in the set $\tilde{E}$ defined by 
    \begin{multline}\label{espace fonctionnel tilde E} \tilde{E}\mathrel{\mathop:}= \bigg\{ (\rho,w,u) \ \bigg| \ \rho\in \mathcal{C}_b(\R^+;\dot B_{2,1}^{\frac{d}{2}}), \ u\in \mathcal{C}_b(\R^+;\dot B_{2,1}^{\frac{d}{2}-1})\cap L^1(\R^+,\dot B_{2,1}^{\frac{d}{2}+1}), \\ w\in \mathcal{C}_b(\R^+; \dot B_{2,1}^{\frac{d}{2}-1}\cap \dot B_{2,1}^{\frac{d}{2}+1})\cap L^1(\R^+;\dot B_{2,1}^{\frac{d}{2}+1})\bigg\}.\end{multline}
In addition, we have the following inequality for any $t\in\R^+$ : \begin{equation}\label{estimée théorème}\mathcal{Z}(t)\leq C \mathcal{Z}_0 \end{equation} where $$\displaylines{\mathcal{Z}(t)\mathrel{\mathop:}= \|\rho\|_{L_t^\infty(\dot B_{2,1}^{\frac{d}{2}})}+\|w\|_{L_t^\infty(\dot B_{2,1}^{\frac{d}{2}-1}\cap \dot B_{2,1}^{\frac{d}{2}+1})}+\|u\|_{L_t^\infty(\dot B_{2,1}^{\frac{d}{2}-1})}+\|(w,u)\|_{L_t^1(\dot B_{2,1}^{\frac{d}{2}+1})} \hfill\cr\hfill +\|w-u\|_{L_t^\infty(\dot B_{2,1}^{\frac{d}{2}-1})\cap L_t^1(\dot B_{2,1}^{\frac{d}{2}-1})}.}$$
    
\end{theorem}

For the study of such a system, the key ingredient is to obtain good a priori global-in-time estimates. Then, very classical arguments lead to the existence and uniqueness of global solutions.

These estimates are obtained by looking at each equation of the \eqref{Euler-Navier-Stokes2} system in the corresponding Besov space. However, a crucial piece of information is missing at low frequencies. To explain this, let us look at the corresponding linearized system:
$$\left\{\begin{array}{l} \partial_t w+w-u=0, \\ \partial_t u-\Delta u=0. \end{array}\right.$$

So, using the Fourier transform, we have:
$$\frac{d}{dt}\begin{pmatrix} \widehat{w} \\ \widehat{u} \end{pmatrix}+\begin{pmatrix} Id & -Id \\ 0 & |\xi|^2 \ Id \end{pmatrix}\begin{pmatrix} \widehat{w} \\ \widehat{u} \end{pmatrix}=0.$$

The eigenvalues of the matrix are $1$ and $|\xi|^2$. In the low-frequency regime, for the equation on $u$, we only capture the eigenvalue $|\xi|^2$, and to recover the missing information, we use what is called the "damped mode", namely $w-u$. From a physical perspective, the quantity $w-u$ represents the relative velocity between the dispersed phase and the carrier fluid. The friction term therefore acts as a relaxation mechanism, tending to align the two velocities as time evolves. Consequently, one expects $w-u$ to decay faster than each velocity field separately, a phenomenon that is reflected in the decay estimates established in Theorem~\ref{estimées de décroissance}.

By obtaining an estimate on $w-u$ in the Besov space of regularity $\frac{d}{2} - 1$ and by properly combining the previous information, we can deduce the estimate \eqref{estimée théorème système 1}, which allows us to conclude the proof of Theorem~\ref{théorème existence et unicité}.

The second step is devoted to obtaining the decay estimates \eqref{taux de décroissance L2}. To this end, we use a classical approach that goes back to the work of J.~Nash~\cite{Nash} on parabolic equations, and which has been recently used, for example, by R.~Danchin in~\cite{Danchin}.

 The idea is as follows. Suppose there exists a non-negative Lyapunov functional $\mathcal{L}$ and a dissipation rate such that
\begin{equation}\label{inégalité Nash}
\frac{d}{dt}\mathcal{L} + \mathcal{H} \leq 0 \quad \text{a.e. on } \ \mathbb{R}^+.
\end{equation}

Suppose further that there exists another lower-order functional $\mathcal{N}$ such that
$$
\mathcal{N} \leq \mathcal{N}_0 \quad \text{a.e. on } \ \mathbb{R}^+.
$$

Finally, assume that $\mathcal{L}$ is an "intermediate" functional between $\mathcal{N}$ and $\mathcal{H}$, in the sense that there exist $\theta \in (0,1)$ and $C > 0$ such that
$$
\mathcal{L} \leq C \mathcal{H}^\theta \mathcal{N}^{1 - \theta} \quad \text{a.e. on } \ \mathbb{R}^+.
$$

Then, plugging this inequality into~\eqref{inégalité Nash}, we obtain
$$
\frac{d}{dt}\mathcal{L} + c_0 \mathcal{L}^{\frac{1}{\theta}} \leq 0, \quad \text{with} \quad c_0 := C^{-\frac{1}{\theta}} \mathcal{N}_0^{1 - \frac{1}{\theta}},
$$
and thus, by time integration,
$$
\mathcal{L}(t) \leq \mathcal{L}(0) \left(1 + \frac{1 - \theta}{\theta} c_0 \mathcal{L}_0^{\frac{1 - \theta}{\theta}} t \right)^{-\frac{\theta}{1 - \theta}}.
$$

In our setting, the functionals $\mathcal{L}$ and $\mathcal{H}$ will appear naturally (and be defined) in the proof of Theorem~\ref{théorème existence et unicité}. For $\mathcal{N}$, we will use the weaker norm $\|\cdot\|_{\dot{B}_{2,\infty}^{-\frac{d}{2}}}$ (instead of the $L^1$ norm).

By assuming more regularity and accordingly adjusting the functionals $\mathcal{L}$ and $\mathcal{H}$, we similarly obtain the estimates~\eqref{taux de décroissance L2 k}.

The derivation of~\eqref{taux de décroissance L2 mode amorti} is carried out "by hand", with a major difficulty arising from controlling $\Delta u$, which requires special care since we are working in a critical framework. As this approach does not hold in dimension $2$, we propose an alternative method for obtaining~\eqref{estimée décroissance u-v 2d}, with a slower decay rate compared to the case of dimensions $d \geq 3$.

\section{Well-posedness}
In this section, we first establish the a priori estimates \eqref{estimée théorème} for the system \eqref{Euler-Navier-Stokes3}. We then prove uniqueness via a stability lemma corresponding to a change of initial conditions, from which we deduce Theorem~\ref{théorème existence et unicité2}, and consequently Theorem~\ref{théorème existence et unicité}.

\subsection{A priori estimates}~\\
Let us next assume we have a smooth enough solution $(\rho, w, u)$ of system \eqref{Euler-Navier-Stokes3} on $[0,T]\times \R^d$. 

From \cite[chap. 3]{BCD}, we easily obtain the estimate on the first equation of \eqref{Euler-Navier-Stokes3} : 
\begin{lemma}\label{équation de transport}
    We have the following inequality on $\rho$ for all $t\in [0,T]$: \begin{eqnarray}\label{estimée sur rho}\|\rho(t)\|_{\dot B_{2,1}^{\frac{d}{2}}}\leq \|\rho_0\|_{\dot B_{2,1}^{\frac{d}{2}}}+C \int_0^t \|\rho\|_{\dot B_{2,1}^{\frac{d}{2}}}\|w\|_{\dot B_{2,1}^{\frac{d}{2}+1}}d\tau.\end{eqnarray}
\end{lemma}
We then deduce by Grönwall's lemma: \begin{eqnarray}\label{estimée sur le terme de transport}
    \|\rho(t)\|_{\dot B_{2,1}^{\frac{d}{2}}}\leq \|\rho_0\|_{\dot B_{2,1}^{\frac{d}{2}}}\exp\left(C \int_0^t \|w\|_{\dot B_{2,1}^{\frac{d}{2}+1}} d\tau \right).
\end{eqnarray}

Suppose there exists $\tilde{T}>0$ such that $\tilde{T}\leq T$ and on $[0,\tilde{T}]$ \begin{eqnarray}\label{hypothèse de petitesse 2}
    \int_0^t \|w\|_{\dot B_{2,1}^{\frac{d}{2}+1}}d\tau \leq C^{-1}\ln(2),
\end{eqnarray}
where $C$ is the constant appearing in Lemma \ref{équation de transport}.

Then, we have for $t\in [0,\tilde{T}]$ : \begin{eqnarray}\label{estimée sur rho gronwall}
    \|\rho(t)\|_{\dot B_{2,1}^{\frac{d}{2}}}\leq 2\|\rho_0\|_{\dot B_{2,1}^{\frac{d}{2}}}.
\end{eqnarray}

Let us study the Navier-Stokes equation and deduce an estimate of its velocity.

\begin{lemma}\label{lemme estimée sur v}
There exists $C>0$ such that for all $t\in [0,\tilde{T}]$:  
\begin{multline}\label{estimée sur v 2} \|u(t)\|_{\dot B_{2,1}^{\frac{d}{2}-1}}+\frac{1}{c_B} \int_0^t \|u\|_{\dot B_{2,1}^{\frac{d}{2}+1}} d\tau \\ \leq \|u_0\|_{\dot B_{2,1}^{\frac{d}{2}-1}} +C\int_0^t \|u\|_{\dot B_{2,1}^{\frac{d}{2}-1}}\|u\|_{\dot B_{2,1}^{\frac{d}{2}+1}}d\tau \\ +C\|\rho_0\|_{\dot B_{2,1}^{\frac{d}{2}}}\int_0^t\|w-u\|_{\dot B_{2,1}^{\frac{d}{2}-1}}d\tau.\end{multline} where $c_B$ has been defined from \eqref{constante de Bernstein}.
\end{lemma}

\begin{proof}
Applying the operator $\dot\Delta_j$ to the last equation of the system \eqref{Euler-Navier-Stokes3} gives: $$ \partial_t u_j-\Delta u_j=-\dot\Delta_j \mathbb{P}(u\cdot \nabla u)+\dot\Delta_j \mathbb{P}(\rho(w-u)).$$

By taking the scalar product with $u_j$, by Cauchy-Schwarz inequality and by continuity of Leray's projector on $L^2$, we have: $$\displaylines{\frac{1}{2}\frac{d}{dt}\|u_j\|_{L^2}^2-\int_{\R^d}\Delta u_j\cdot u_j dx\leq \big(\|\dot\Delta_j (u\cdot \nabla u)\|_{L^2}+\|\dot\Delta_j (\rho(w-u))\|_{L^2} \big) \|u_j\|_{L^2}.}$$

Now, by integration by parts, we have: $$-\int_{\R^d}\Delta u_j\cdot u_j dx=\int_{\R^d}|\nabla u_j|^2 dx=\|\nabla u_j\|_{L^2}^2.$$

By \eqref{constante de Bernstein}, we have: $$\frac{1}{\sqrt{c_B}} 2^{j}\|u_j\|_{L^2} \leq \|\nabla u_j\|_{L^2}.$$

By Lemma \ref{lemme edo}, by multiplying by $2^{j(\frac{d}{2}-1)}$ and summing over $j\in\Z$, we obtain: $$\displaylines{\|u(t)\|_{\dot B_{2,1}^{\frac{d}{2}-1}}+\frac{1}{c_B} \int_0^t \|u\|_{\dot B_{2,1}^{\frac{d}{2}+1}}d\tau \leq \|u_0\|_{\dot B_{2,1}^{\frac{d}{2}+1}}d\tau+\int_0^t \|(u\cdot \nabla)u\|_{\dot B_{2,1}^{\frac{d}{2}-1}}\hfill\cr\hfill+\int_0^t \|\rho(w-u)\|_{\dot B_{2,1}^{\frac{d}{2}-1}}d\tau.}$$

Now, by the product laws of Lemma \ref{Produit espace de Besov}, we have: $$\left\{ \begin{array}{l}\|u\cdot \nabla u\|_{\dot B_{2,1}^{\frac{d}{2}-1}}\lesssim \|u\|_{\dot B_{2,1}^{\frac{d}{2}-1}}\|\nabla u\|_{\dot B_{2,1}^{\frac{d}{2}}}\lesssim \|u\|_{\dot B_{2,1}^{\frac{d}{2}-1}}\|u\|_{\dot B_{2,1}^{\frac{d}{2}+1}},
\\ \|\rho(w-u)\|_{\dot B_{2,1}^{\frac{d}{2}-1}}\lesssim \|\rho\|_{\dot B_{2,1}^{\frac{d}{2}}} \|w-u\|_{\dot B_{2,1}^{\frac{d}{2}-1}}.
\end{array}\right.$$

Hence remembering \eqref{estimée sur rho gronwall} the estimate of the lemma.
\end{proof}

Now, let us establish an estimate of the velocity $w$ of the Euler equation.

\begin{lemma}\label{estimée sur u lemme}
    We have the following inequality for any $t\in [0,T]$ :
    \begin{multline}\label{estimée sur u}\|w(t)\|_{\dot B_{2,1}^{\frac{d}{2}+1}}+\int_0^t \|w\|_{\dot B_{2,1}^{\frac{d}{2}+1}}d\tau \leq \|w_0\|_{\dot B_{2,1}^{\frac{d}{2}+1}}+C\int_0^t \|w\|_{\dot B_{2,1}^{\frac{d}{2}+1}}^2 d\tau \\+\int_0^t \|u\|_{\dot B_{2,1}^{\frac{d}{2}+1}}d\tau.\end{multline}
    
\end{lemma}

\begin{proof}
    By applying $\dot\Delta_j$ to the second equation of \eqref{Euler-Navier-Stokes2}, we get: $$\partial_t w_j+w_j=u_j-(w\cdot \nabla)w_j-\left(\dot\Delta_j(w\cdot \nabla)w-(w\cdot \nabla) w_j \right).$$

    By taking the scalar product with $w_j$, then we have: $$\displaylines{\frac{1}{2}\frac{d}{dt}\|w_j\|_{L^2}^2+\|w_j\|_{L^2}^2=\int_{\R^d}u_j\cdot w_j dx-\int_{\R^d} (w\cdot \nabla)w_j \cdot w_j dx \hfill\cr\hfill -\int_{\R^d} [\dot\Delta_j,w\cdot \nabla]w\cdot w_j dx.}$$

    However, we have : $$-\int_{\R^d}(w\cdot \nabla)w_j\cdot w_j dx=\frac{1}{2}\int_{\R^d} \dive(w) |w_j|^2 dx \leq \frac{1}{2}\|\dive w\|_{L^\infty}\|w_j\|_{L^2}^2.$$

    Then, by Cauchy-Schwarz inequality, we obtain: 
    $$\displaylines{\frac{1}{2}\frac{d}{dt}\|w_j\|_{L^2}^2+\|w_j\|_{L^2}^2\leq \left(\|u_j\|_{L^2}+\|[\dot\Delta_j,w\cdot\nabla]w\|_{L^2}+\|\dive w\|_{L^\infty}\|w_j\|_{L^2}\right) \hfill\cr\hfill\times \|w_j\|_{L^2}.}$$

    By Lemma \ref{lemme edo}, we get: $$\displaylines{\|w_j(t)\|_{L^2}+\int_0^t \|w_j\|_{L^2}d\tau \leq  \|w_{j,0}\|_{L^2}+\int_0^t \|[\dot\Delta_j,w\cdot\nabla]w\|_{L^2}d\tau \hfill\cr\hfill +\int_0^t \|\dive w\|_{L^\infty}\|w_j\|_{L^2}d\tau+\int_0^t \|u_j\|_{L^2}d\tau.}$$

    Now, by Lemma \ref{commutateur}, the injection $\dot B_{2,1}^{\frac{d}{2}}\hookrightarrow L^\infty$, by multiplying by $2^{j(\frac{d}{2}+1)}$ then by summing over $j\in\Z$, we obtain \eqref{estimée sur u}.
\end{proof}

To close the a priori estimates, we need information at low frequencies, given by the damped mode $w-u$, which verifies a better decay than $u$.

\begin{lemma}\label{lemme estimée sur w-u}
    We have the following estimate on $w-u$ for $t\in [0,\tilde{T}]$ where $\tilde{T}$ is defined from \eqref{hypothèse de petitesse 2} and the constant $c_B$ by \eqref{constante de Bernstein}: 
\begin{multline}\label{estimée sur u-v 2}\|(w-u)(t)\|_{\dot B_{2,1}^{\frac{d}{2}-1}}+\int_0^t \|w-u\|_{\dot B_{2,1}^{\frac{d}{2}-1}} d\tau \\ \leq \|w_0-u_0\|_{\dot B_{2,1}^{\frac{d}{2}-1}}+c_B\int_0^t \|u\|_{\dot B_{2,1}^{\frac{d}{2}+1}}d\tau +C \|\rho_0\|_{\dot B_{2,1}^{\frac{d}{2}}} \int_0^t \|w-u\|_{\dot B_{2,1}^{\frac{d}{2}-1}}d\tau \\ +C\int_0^t \|w-u\|_{\dot B_{2,1}^{\frac{d}{2}-1}}\|w\|_{\dot B_{2,1}^{\frac{d}{2}+1}}d\tau \\ +C\int_0^t \|u\|_{\dot B_{2,1}^{\frac{d}{2}-1}}\left(\|w\|_{\dot B_{2,1}^{\frac{d}{2}+1}}+\|u\|_{\dot B_{2,1}^{\frac{d}{2}+1}}\right)d\tau. \end{multline}
\end{lemma}
\begin{proof}
    We have that $w-u$ satisfies the following equation by \eqref{Euler-Navier-Stokes2} : 
    $$\partial_t (w-u)+w-u=-\Delta u-\mathbb{P}(\rho (w-u))+\mathbb{P}(u\cdot \nabla) u-(w\cdot \nabla)w.$$

By applying $\dot\Delta_j$, we obtain: 
$$\partial_t (w_j-u_j)+w_j-u_j=-\Delta u_j-\dot\Delta_j\mathbb{P}\rho (w-u)+\dot\Delta_j\left(\mathbb{P}(u\cdot \nabla) u-(w\cdot \nabla)w\right).$$

Taking the scalar product with $w_j-u_j$, then by the Cauchy-Schwarz inequality and the continuity of the Leray projector on $L^2$, we have: $$\displaylines{\frac{1}{2}\frac{d}{dt}\|w_j-u_j\|_{L^2}^2+\|w_j-u_j\|_{L^2}^2\leq \bigg(\|\Delta u_j\|_{L^2}+\|\dot\Delta_j(\rho(w-u))\|_{L^2}\hfill\cr\hfill+\|\dot\Delta_j\left(\mathbb{P}(u\cdot\nabla)u-(w\cdot\nabla)w\right)\|_{L^2}\bigg)\times\|w_j-u_j\|_{L^2}.}$$

The last term of the previous inequality can be written as follows:
$$(w\cdot\nabla)w-\mathbb{P}\left((u\cdot\nabla)u\right)=\left((w-u)\cdot\nabla\right)w+(u\cdot \nabla)(w-u)+(Id-\mathbb{P})\left((u\cdot\nabla)u\right).$$

Applying Lemma \ref{lemme edo}, then multiplying by $\displaystyle 2^{j\left({\frac{d}{2}-1}\right)}$ and summing over $j\in\Z$, we get: 
$$\displaylines{\|(w-u)(t)\|_{\dot B_{2,1}^{\frac{d}{2}-1}}+\int_0^t \|w-u\|_{\dot B_{2,1}^{\frac{d}{2}-1}}d\tau \hfill\cr \leq \|(w_{0}-u_{0})\|_{\dot B_{2,1}^{\frac{d}{2}-1}}+\int_0^t\bigg(\|\Delta u\|_{\dot B_{2,1}^{\frac{d}{2}-1}} +\|\rho(w-u)\|_{\dot B_{2,1}^{\frac{d}{2}-1}}\bigg)d\tau \hfill\cr\hfill+\int_0^t \|(w-u)\cdot\nabla w\|_{\dot B_{2,1}^{\frac{d}{2}-1}} d\tau+\int_0^t \|u\cdot\nabla(w-u)\|_{\dot B_{2,1}^{\frac{d}{2}-1}} d\tau\cr\hfill +\int_0^t \|(Id-\mathbb{P})(u\cdot\nabla)u\|_{\dot B_{2,1}^{\frac{d}{2}-1}} d\tau.}$$
Let us estimate the terms on the right-hand side of the previous inequality.

By the product laws of Lemma \ref{Produit espace de Besov}, we have:
\begin{itemize}

    \item[$\bullet$]  $\displaystyle \int_0^t \|\rho(w-u)\|_{\dot B_{2,1}^{\frac{d}{2}-1}}d\tau\lesssim \|\rho_0\|_{\dot B_{2,1}^{\frac{d}{2}}}\int_0^t \|w-u\|_{\dot B_{2,1}^{\frac{d}{2}-1}}d\tau$ by \eqref{estimée sur rho gronwall},

    \item[$\bullet$] $\displaystyle \int_0^t \|(w-u)\cdot \nabla w\|_{\dot B_{2,1}^{\frac{d}{2}-1}}d\tau\lesssim \int_0^t \|w-u\|_{\dot B_{2,1}^{\frac{d}{2}-1}}\|w\|_{\dot B_{2,1}^{\frac{d}{2}+1}}d\tau$.

    \item[$\bullet$] By triangular inequality, we also have: \begin{align*}\int_0^t \|(u\cdot\nabla)(w-u)\|_{\dot B_{2,1}^{\frac{d}{2}-1}}d\tau & \lesssim \int_0^t \|u\cdot \nabla w\|_{\dot B_{2,1}^{\frac{d}{2}-1}}d\tau +\int_0^t \|u\cdot \nabla u\|_{\dot B_{2,1}^{\frac{d}{2}-1}} d\tau \\ & \lesssim \int_0^t \|u\|_{\dot B_{2,1}^{\frac{d}{2}-1}}(\|w\|_{\dot B_{2,1}^{\frac{d}{2}+1}}+\|u\|_{\dot B_{2,1}^{\frac{d}{2}+1}})d\tau.
    \end{align*}
\item[$\bullet$] Since $Id-\mathbb{P}$ is a homogeneous Fourier multiplier of order 0, we then have by the product laws of Lemma \ref{Produit espace de Besov} : $$\int_0^t \|(Id-\mathbb{P})(u\cdot\nabla)u\|_{\dot B_{2,1}^{\frac{d}{2}-1}} d\tau \lesssim \int_0^t \|u\|_{\dot B_{2,1}^{\frac{d}{2}-1}}\|u\|_{\dot B_{2,1}^{\frac{d}{2}+1}} d\tau.$$
\end{itemize}
We then deduce the inequality of the proposition.
\end{proof}

By multiplying by $\frac{1}{4c_B}$ the inequality \eqref{estimée sur u}, by $\frac{1}{2c_B^2}$ the inequality \eqref{estimée sur u-v 2} and summing these inequalities with \eqref{estimée sur v 2}, we obtain for any $t\in[0,\tilde{T}]$ : $$\displaylines{\|u(t)\|_{\dot B_{2,1}^{\frac{d}{2}-1}}+\frac{1}{4c_B}\|w(t)\|_{\dot B_{2,1}^{\frac{d}{2}+1}}+\frac{1}{2c_B^2}\|(w-u)(t)\|_{\dot B_{2,1}^{\frac{d}{2}-1}} \hfill\cr +\int_0^t \bigg(\frac{1}{c_B} \|u\|_{\dot B_{2,1}^{\frac{d}{2}+1}}+\frac{1}{4c_B}\|w\|_{\dot B_{2,1}^{\frac{d}{2}+1}}+\frac{1}{2c_B^2}\|w-u\|_{\dot B_{2,1}^{\frac{d}{2}-1}}\bigg)d\tau 
\hfill \cr \leq \frac{1}{4 c_B}\|w_0\|_{\dot B_{2,1}^{\frac{d}{2}+1}}+\|u_0\|_{\dot B_{2,1}^{\frac{d}{2}-1}}+\frac{1}{2 c_B^2}\|w_0-u_0\|_{\dot B_{2,1}^{\frac{d}{2}-1}} \hfill\cr\hfill +\bigg(\frac{1}{4c_B}+\frac{1}{2 c_B}\bigg)\int_0^t \|u\|_{\dot B_{2,1}^{\frac{d}{2}+1}}d\tau+\bigg(1+\frac{1}{2 c_B^2}\bigg)2C \|\rho_0\|_{\dot B_{2,1}^{\frac{d}{2}}}\int_0^t \|w-u\|_{\dot B_{2,1}^{\frac{d}{2}-1}}d\tau
\hfill\cr\hfill +\tilde{C}\int_0^t \bigg(\|w\|_{\dot B_{2,1}^{\frac{d}{2}+1}}^2+\|w-u\|_{\dot B_{2,1}^{\frac{d}{2}-1}}\|w\|_{\dot B_{2,1}^{\frac{d}{2}+1}}+\|u\|_{\dot B_{2,1}^{\frac{d}{2}-1}}(\|w\|_{\dot B_{2,1}^{\frac{d}{2}+1}} \hfill\cr\hfill +\|u\|_{\dot B_{2,1}^{\frac{d}{2}+1}})\bigg)d\tau.}$$

Let us set: \begin{equation}\label{définition de L}\mathcal{L}(t)\mathrel{\mathop:}=\|u(t)\|_{\dot B_{2,1}^{\frac{d}{2}-1}}+\frac{1}{4 c_B}\|w(t)\|_{\dot B_{2,1}^{\frac{d}{2}+1}}+\frac{1}{2 c_B^2}\|(w-u)(t)\|_{\dot B_{2,1}^{\frac{d}{2}-1}}, \end{equation} and \begin{equation}\label{définition de H}\mathcal{H}(t)\mathrel{\mathop:}= \frac{1}{4c_B}\|u(t)\|_{\dot B_{2,1}^{\frac{d}{2}+1}}+\frac{1}{4 c_B}\|w(t)\|_{\dot B_{2,1}^{\frac{d}{2}+1}}+\frac{1}{4 c_B^2}\|(w-u)(t)\|_{\dot B_{2,1}^{\frac{d}{2}-1}}.\end{equation}

By assuming $\left(2+\frac{1}{c_B^2}\right)C \|\rho_0\|_{\dot B_{2,1}^{\frac{d}{2}}}\leq \frac{1}{4c_B^2}$, we have then: $$\mathcal{L}(t)+\int_0^t \mathcal{H}(\tau)d\tau \leq \mathcal{L}(0)+C\int_0^t \mathcal{L}(\tau)\mathcal{H}(\tau) d\tau.$$

\begin{prop}
    If we take $\mathcal{L}(0)$ and $\|\rho\|_{\dot B_{2,1}^{\frac{d}{2}}}$ to be sufficiently small, we have for all $t\in [0,\tilde{T}]$ : $$\mathcal{L}(t)+\frac{1}{2}\int_0^t \mathcal{H}(\tau)d\tau \leq \mathcal{L}(0).$$
\end{prop}

On $w$, we can deduce a better estimate thanks to the estimates on $u$ and $w-u$. In fact, we have: $$\|w\|_{\dot B_{2,1}^{\frac{d}{2}-1}}=\|w-u+u\|_{\dot B_{2,1}^{\frac{d}{2}-1}}\leq \|w-u\|_{\dot B_{2,1}^{\frac{d}{2}-1}}+\|u\|_{\dot B_{2,1}^{\frac{d}{2}-1}}\lesssim \mathcal{L}(0).$$

We then deduce the estimates of Theorem \ref{estimée théorème}.

Clearly $t=0$ plays no particular role, and we can apply the same argument to any sub-interval of $[0,\tilde{T}]$, which leads to : \begin{eqnarray}\label{inégalité fonctionnelle}\mathcal{L}(t)+\frac{1}{2}\int_{t_0}^t \mathcal{H}(\tau)d\tau \leq \mathcal{L}(t_0), \quad 0\leq t_0\leq t \leq T.\end{eqnarray}

To conclude on the a priori estimates of Theorem \ref{théorème existence et unicité}, we still have to estimate the pressure term.

To achieve this, we just use \eqref{expression de la pression P}, the fact that $(Id-\mathbb{P})$ is a homogeneous Fourier multiplier of order $0$ and the product laws of Lemma \ref{Produit espace de Besov}. In the end, we get:
\begin{align*}  \int_0^t \|\nabla P\|_{\dot B_{2,1}^{\frac{d}{2}-1}} d\tau & \lesssim \int_0^t \|(u\cdot \nabla) u\|_{\dot B_{2,1}^{\frac{d}{2}-1}}d\tau +\int_0^t \|\rho(w-u)\|_{\dot B_{2,1}^{\frac{d}{2}-1}} d\tau \\ & \lesssim \int_0^t \|u\|_{\dot B_{2,1}^{\frac{d}{2}-1}}\|u\|_{\dot B_{2,1}^{\frac{d}{2}+1}}d\tau+\int_0^t \|\rho\|_{\dot B_{2,1}^{\frac{d}{2}}}\|w-u\|_{\dot B_{2,1}^{\frac{d}{2}-1}} d\tau  \\ & \lesssim \mathcal{L}(0).
\end{align*}

\subsection{Stability estimates and uniqueness}

First, note that we have the following lemma:
\begin{lemma}
    Let $(\rho,w,u)$ be a solution of the system \eqref{Euler-Navier-Stokes3} belonging to $\tilde{E}$ defined in \eqref{espace fonctionnel tilde E}. Then, we have : $$\rho(t)-\rho_0\in \mathcal{C}(\R^+;\dot B_{2,1}^{\frac{d}{2}-1}) \quad \text{and} \quad w\in \mathcal{C}_b(\R_+; \dot B_{2,1}^{\frac{d}{2}}).$$
\end{lemma}
\begin{proof}
The second property comes from the interpolation inequality in Besov spaces and the fact that the solution $w$ is in $\mathcal{C}_b(\R^+;\dot B_{2,1}^{\frac{d}{2}-1}\cap \dot B_{2,1}^{\frac{d}{2}+1})$.

For the first property, let us first note that $$\|\partial_t \rho\|_{\dot B_{2,1}^{\frac{d}{2}-1}}=\|\dive(\rho w)\|_{\dot B_{2,1}^{\frac{d}{2}-1}}\lesssim \|\rho w\|_{\dot B_{2,1}^{\frac{d}{2}}}\lesssim \|\rho\|_{\dot B_{2,1}^{\frac{d}{2}}}\|w\|_{\dot B_{2,1}^{\frac{d}{2}}},$$ by the product laws of Lemma \ref{Produit espace de Besov} and the fact that $\partial_t \rho+\dive(\rho w)=0$.

  We therefore have that $\partial_t \rho \in L^\infty(\R^+; \dot B_{2,1}^{\frac{d}{2}-1})\subset L_{loc}^1(\R^+;\dot B_{2,1}^{\frac{d}{2}-1})$. We then deduce the first property by integration.
\end{proof}

\begin{remark}
This lemma is useful because we are going to study the uniqueness for $\rho-\rho_0$ in $\mathcal{C}(\R^+;\dot B_{2,1}^{\frac{d}{2}-1})$ and $w$ in $\mathcal{C}_b(\R^+;\dot B_{2,1}^{\frac{d}{2}})$.
\end{remark}

Let us prove a stability lemma for Euler-Navier-Stokes solutions with respect to initial conditions.
\begin{lemma}\label{lemme pour unicité}
    Let $(\rho_1,w_1,u_1)$ and $(\rho_2,w_2,u_2)$ be two solutions of \eqref{Euler-Navier-Stokes3} with initial data $(\rho_{1,0},w_{1,0},u_{1,0})$ and $(\rho_{2,0},w_{2,0},u_{2,0})$ respectively and belonging to $E$ defined in \eqref{espace fonctionnel E}, we have the following inequalities on $$(\delta \rho, \delta w, \delta u) \mathrel{\mathop:}= (\rho_2-\rho_2,w_2-w_1,u_2-u_1)$$ for all $t\in\R^+$ where $c_B$ is the constant in Lemma \ref{lemme estimée sur v}: 
    \begin{multline}\label{unicité estimée rho}\|\delta \rho(t)\|_{\dot B_{2,1}^{\frac{d}{2}-1}}\leq  \|\delta \rho_0\|_{\dot B_{2,1}^{\frac{d}{2}-1}}+C\int_0^t \bigg(\|\delta \rho\|_{\dot B_{2,1}^{\frac{d}{2}-1}}\|w_2\|_{\dot B_{2,1}^{\frac{d}{2}+1}} \\ +\|\rho_1\|_{\dot B_{2,1}^{\frac{d}{2}}}\|\delta w\|_{\dot B_{2,1}^{\frac{d}{2}}} \bigg) d\tau,
    \end{multline}
    \begin{multline}\label{unicité estimée v}
    \|\delta u(t)\|_{\dot B_{2,1}^{\frac{d}{2}-1}}+ \frac{1}{c_B}\int_0^t \|\delta u\|_{\dot B_{2,1}^{\frac{d}{2}+1}} \\ \leq \|\delta u_0\|_{\dot B_{2,1}^{\frac{d}{2}-1}} +C\int_0^t \|\delta u\|_{\dot B_{2,1}^{\frac{d}{2}-1}}\|(u_1,u_2)\|_{\dot B_{2,1}^{\frac{d}{2}+1}} d\tau \\ +C\int_0^t \|\delta \rho\|_{\dot B_{2,1}^{\frac{d}{2}-1}}\|w_2-u_2\|_{\dot B_{2,1}^{\frac{d}{2}}} d\tau  +C\int_0^t \|\rho_1\|_{\dot B_{2,1}^{\frac{d}{2}}}\|\delta w-\delta u\|_{\dot B_{2,1}^{\frac{d}{2}-1}}d\tau,\end{multline}
\begin{multline}\label{unicité estimée u}\|\delta w\|_{\dot B_{2,1}^{\frac{d}{2}}}+\int_0^t \|\delta w\|_{\dot B_{2,1}^{\frac{d}{2}}}^h d\tau\leq \|\delta w_0\|_{\dot B_{2,1}^{\frac{d}{2}}}+\int_0^t \bigg(\|\delta u\|_{\dot B_{2,1}^{\frac{d}{2}}}^h \\ +\|\delta w-\delta u\|_{\dot B_{2,1}^{\frac{d}{2}}}^l \bigg) d\tau +C\int_0^t \|\delta w\|_{\dot B_{2,1}^{\frac{d}{2}}}\|(w_1,w_2)\|_{\dot B_{2,1}^{\frac{d}{2}+1}}d\tau,\end{multline}
    \begin{multline}\label{unicité estimée u-v}\|(\delta w-\delta u)(t)\|_{\dot B_{2,1}^{\frac{d}{2}-1}}+\int_0^t \|\delta w-\delta u\|_{\dot B_{2,1}^{\frac{d}{2}-1}}d\tau \\ \leq \|\delta w_0-\delta u_0\|_{\dot B_{2,1}^{\frac{d}{2}+1}} +\int_0^t \|\delta u\|_{\dot B_{2,1}^{\frac{d}{2}+1}}(c_B+C\|u_1\|_{\dot B_{2,1}^{\frac{d}{2}-1}})d\tau \cr\hfill+C\int_0^t \bigg(\|\delta \rho\|_{\dot B_{2,1}^{\frac{d}{2}-1}}+\|\delta u\|_{\dot B_{2,1}^{\frac{d}{2}-1}}\bigg)\|w_2-u_2\|_{\dot B_{2,1}^{\frac{d}{2}}} d\tau \cr\hfill+C\int_0^t \bigg(\|\rho_1\|_{\dot B_{2,1}^{\frac{d}{2}}}+\|w_2\|_{\dot B_{2,1}^{\frac{d}{2}+1}}\bigg)\|\delta w-\delta u\|_{\dot B_{2,1}^{\frac{d}{2}-1}} d\tau \cr\hfill+C\int_0^t \|u_1\|_{\dot B_{2,1}^{\frac{d}{2}+1}}\|\delta w-\delta u\|_{\dot B_{2,1}^{\frac{d}{2}-1}} d\tau  +C\int_0^t \|w_1-u_1\|_{\dot B_{2,1}^{\frac{d}{2}}} \|\delta w\|_{\dot B_{2,1}^{\frac{d}{2}}}d\tau \cr\hfill + C\int_0^t \|\delta u\|_{\dot B_{2,1}^{\frac{d}{2}-1}}\|(w_2,u_2)\|_{\dot B_{2,1}^{\frac{d}{2}+1}} d\tau.\end{multline}
\end{lemma}

\begin{proof}
   We have that $(\delta \rho, \delta w,\delta u)$ verifies the following system: 
    \begin{equation}\label{système pour unicité}\left\{\begin{array}{l}
    \partial_t \delta \rho +\dive(\delta \rho \ w_2)+\dive(\rho_1 \delta w)=0,      
    \\
    \partial_t \delta w+(\delta w\cdot \nabla)w_2+(w_1\cdot \nabla)\delta w+\delta w-\delta u=0,
    \\
    \partial_t \delta u+\mathbb{P}\left((\delta u\cdot \nabla)u_2\right)+\mathbb{P}\left((u_1\cdot \nabla)\delta u\right) \\ \hfill =\Delta \delta u+ \mathbb{P} \left(\delta \rho(w_2-u_2)\right)  +\mathbb{P}\left(\rho_1 (\delta w-\delta u)\right).
    \end{array} \right.\end{equation}

Let us prove the estimates of the proposition one by one:

\begin{enumerate}
    \item For the first equation of the system, we can argue as in the proof of the lemma \ref{équation de transport} by putting in source term $\dive(\rho_1 \delta w).$ We can then easily obtain: $$\displaylines{\|\delta \rho(t)\|_{\dot B_{2,1}^{\frac{d}{2}-1}}\leq  \|\delta \rho_0\|_{\dot B_{2,1}^{\frac{d}{2}-1}}+C\int_0^t \bigg(\|\delta \rho\|_{\dot B_{2,1}^{\frac{d}{2}-1}}\|w_2\|_{\dot B_{2,1}^{\frac{d}{2}+1}} \hfill\cr\hfill +\|\dive(\rho_1 \delta w)\|_{\dot B_{2,1}^{\frac{d}{2}-1}} \bigg) d\tau.}$$
By product laws of Lemma \ref{Produit espace de Besov}, we then have: $$\|\dive(\rho_1 \delta w)\|_{\dot B_{2,1}^{\frac{d}{2}-1}}\lesssim \|\rho_1 \delta w\|_{\dot B_{2,1}^{\frac{d}{2}}}\lesssim \|\rho_1\|_{\dot B_{2,1}^{\frac{d}{2}}}\|\delta w\|_{\dot B_{2,1}^{\frac{d}{2}}},$$
  and deduce \eqref{unicité estimée rho}.
\item By applying the localization operator $\dot\Delta_j$ to the third equation of the system \eqref{système pour unicité}, taking the scalar product with $\delta u_j$, using the fact that the Leray operator is symmetrical on $L^2$ and $\mathbb{P} \delta u_j=\delta u_j$, we obtain: $$\displaylines{\frac{1}{2}\frac{d}{dt}\|\delta u_j\|_{L^2}^2-\int_{\R^d} \Delta \delta u_j\cdot \delta u_j dx \hfill\cr\hfill =-\int_{\R^d} \dot\Delta_j(\delta u\cdot \nabla)u_2\cdot \delta u_j dx -\int_{\R^d}\dot\Delta_j(u_1\cdot \nabla)\delta u\cdot \delta u_j dx \hfill\cr\hfill-\int_{\R^d} \dot\Delta_j\left(\delta \rho(w_2-u_2)\right)\cdot \delta u_j dx -\int_{\R^d} \dot\Delta_j\left(\rho_1(\delta w-\delta u)\right)\cdot \delta u_j dx.}$$

Now, we have by integration by parts and inequality \eqref{constante de Bernstein} : $$-\int_{\R^d} \Delta \delta u_j\cdot \delta u_j dx=\|\nabla \delta u_j\|_{L^2}^2\geq  \frac{1}{c_B} 2^{2j}\|u_j\|_{L^2}^2.$$

In addition, we also have: $$\dot\Delta_j(u_1\cdot \nabla)\delta u=(u_1\cdot \nabla)\delta u_j+[\dot\Delta_j,u_1\cdot \nabla]\delta u.$$

By the condition $\dive u_1=0$, we have by integral by parts: $$\int_{\R^d} (u_1\cdot \nabla)\delta u_j\cdot \delta u_j dx=0.$$

By commutator estimates of Lemma \ref{commutateur}, by the Cauchy-Schwarz inequality and then by Lemma \ref{lemme edo}, by multiplying by $2^{j(\frac{d}{2}-1)}$ and summing over $j\in\Z$, we obtain: $$\displaylines{\|\delta u\|_{\dot B_{2,1}^{\frac{d}{2}-1}}+\frac{1}{c_B}\int_0^t \|\delta u\|_{B_{2,1}^{\frac{d}{2}+1}}d\tau \leq \|\delta u_0\|_{\dot B_{2,1}^{\frac{d}{2}-1}}+\int_0^t \|(\delta u\cdot \nabla) u_2\|_{\dot B_{2,1}^{\frac{d}{2}-1}}d\tau \hfill\cr\hfill+C\int_0^t \|u_1\|_{\dot B_{2,1}^{\frac{d}{2}+1}}\|\delta u\|_{\dot B_{2,1}^{\frac{d}{2}-1}} d\tau+ \int_0^t \|\delta \rho(w_2-u_2)\|_{\dot B_{2,1}^{\frac{d}{2}-1}} d\tau \cr\hfill+\int_0^t \|\rho_1(\delta w-\delta u)\|_{\dot B_{2,1}^{\frac{d}{2}-1}}d\tau.}$$
By product laws of Lemma \ref{Produit espace de Besov}, we then have $$\|(\delta u\cdot\nabla)u_2\|_{\dot B_{2,1}^{\frac{d}{2}-1}}\lesssim \|\delta u\|_{\dot B_{2,1}^{\frac{d}{2}-1}}\|\nabla u_2\|_{\dot B_{2,1}^{\frac{d}{2}}}\lesssim \|\delta u\|_{\dot B_{2,1}^{\frac{d}{2}-1}}\|u_2\|_{\dot B_{2,1}^{\frac{d}{2}+1}}$$ and deduce \eqref{unicité estimée v}.
\item Putting in source term $(\delta w\cdot \nabla)w_2$, separating low and high frequencies for the linear term $w-u$ and following the same procedure as in the proof of the lemma \ref{estimée sur u lemme} with here the use of the commutator estimate of Lemma \ref{commutateur} in $\dot B_{2,1}^{\frac{d}{2}}$, we have: 
$$\displaylines{\|\delta w(t)\|_{\dot B_{2,1}^{\frac{d}{2}}}+\int_0^t \|\delta w\|_{\dot B_{2,1}^{\frac{d}{2}}}^h d\tau \hfill\cr\hfill \leq \|\delta w_0\|_{\dot B_{2,1}^{\frac{d}{2}}}+\int_0^t \left(\|\delta w-\delta u\|_{\dot B_{2,1}^{\frac{d}{2}}}^l+\|\delta u\|_{\dot B_{2,1}^{\frac{d}{2}}}^h\right) d\tau \hfill\cr\hfill+C\int_0^t \|\delta w\|_{\dot B_{2,1}^{\frac{d}{2}}}\|w_1\|_{\dot B_{2,1}^{\frac{d}{2}+1}}d\tau +\int_0^t \|(\delta w\cdot \nabla)w_2\|_{\dot B_{2,1}^{\frac{d}{2}}}d\tau.}$$ 

Now, according to the product laws, we have : $$\|(\delta w\cdot \nabla)w_2\|_{\dot B_{2,1}^{\frac{d}{2}}}\lesssim \|\delta w\|_{\dot B_{2,1}^{\frac{d}{2}}}\|\nabla w_2\|_{\dot B_{2,1}^{\frac{d}{2}}}\lesssim\|\delta w\|_{\dot B_{2,1}^{\frac{d}{2}}}\|w_2\|_{\dot B_{2,1}^{\frac{d}{2}+1}}.$$
We deduce \eqref{unicité estimée u}.

\item By subtracting the third equation from the second in the system \eqref{système pour unicité}, we obtain the following equation: 
$$\displaylines{\partial_t (\delta w-\delta u)+\delta w-\delta u=-\Delta \delta u-\mathbb{P}\delta \rho(w_2-u_2)-\mathbb{P}\left(\rho_1(\delta w-\delta u)\right) \hfill\cr\hfill-\mathbb{P}\left((\delta w-\delta u)\cdot \nabla\right) w_2 -(\delta u\cdot \nabla)(w_2-u_2)-(Id-\mathbb{P})\left((\delta u\cdot \nabla)u_2\right) \cr\hfill-\left((w_1-u_1)\cdot \nabla\right)\delta w - (u_1\cdot \nabla)(\delta w-\delta u)-(Id-\mathbb{P})\left((u_1\cdot \nabla)\delta u\right).}$$

Let us estimate the nonlinear terms, except for the term $(u_1\cdot \nabla)(\delta w-\delta u)$ which will be treaten after, in the space $\dot B_{2,1}^{\frac{d}{2}-1}$ by using the product laws from Lemma \ref{Produit espace de Besov}.
\begin{itemize}
    \item[$\bullet$] $\|\mathbb{P}\delta\rho (w_2-u_2)\|_{\dot B_{2,1}^{\frac{d}{2}-1}}\lesssim \|\delta\rho (w_2-u_2)\|_{\dot B_{2,1}^{\frac{d}{2}-1}}\lesssim\|\delta\rho\|_{\dot B_{2,1}^{\frac{d}{2}-1}} \|w_2-u_2\|_{\dot B_{2,1}^{\frac{d}{2}}}$
    \item[$\bullet$] $\|\mathbb{P}\rho_1 (\delta w-\delta u)\|_{\dot B_{2,1}^{\frac{d}{2}-1}}\lesssim \|\rho_1 (\delta w-\delta u)\|_{\dot B_{2,1}^{\frac{d}{2}-1}}\lesssim\|\rho_1\|_{\dot B_{2,1}^{\frac{d}{2}}} \|\delta w-\delta u\|_{\dot B_{2,1}^{\frac{d}{2}-1}}$
    \item[$\bullet$]$\|\mathbb{P}\left((\delta w-\delta u)\cdot \nabla\right) w_2\|_{\dot B_{2,1}^{\frac{d}{2}-1}}\lesssim \|(\delta w-\delta u)\cdot \nabla w_2\|_{\dot B_{2,1}^{\frac{d}{2}-1}}\\ \lesssim \|\delta w-\delta u\|_{\dot B_{2,1}^{\frac{d}{2}-1}}\| \nabla w_2\|_{\dot B_{2,1}^{\frac{d}{2}}}  \lesssim \|\delta w-\delta u\|_{\dot B_{2,1}^{\frac{d}{2}-1}}\|w_2\|_{\dot B_{2,1}^{\frac{d}{2}+1}}$
    \item[$\bullet$] $\|(\delta u\cdot \nabla)(w_2-u_2)\|_{\dot B_{2,1}^{\frac{d}{2}-1}}\lesssim \|\delta u\|_{\dot B_{2,1}^{\frac{d}{2}-1}}\|\nabla(w_2-u_2)\|_{\dot B_{2,1}^\frac{d}{2}} \\ \lesssim\|\delta u\|_{\dot B_{2,1}^{\frac{d}{2}-1}}\|(w_2,u_2)\|_{\dot B_{2,1}^{\frac{d}{2}+1}} $
    \item[$\bullet$] Similarly to the previous term, we have: $$\displaylines{\|(Id-\mathbb{P})\left((\delta u\cdot \nabla)u_2\right) \|_{\dot B_{2,1}^{\frac{d}{2}-1}}\lesssim\|(\delta u\cdot \nabla)u_2 \|_{\dot B_{2,1}^{\frac{d}{2}-1}} \lesssim \|\delta u\|_{\dot B_{2,1}^{\frac{d}{2}-1}}\|u_2\|_{\dot B_{2,1}^{\frac{d}{2}+1}} .}$$
    \item[$\bullet$] $\|\left((w_1-u_1)\cdot \nabla\right)\delta w\|_{\dot B_{2,1}^{\frac{d}{2}-1}}\lesssim \|w_1-u_1\|_{\dot B_{2,1}^{\frac{d}{2}}}\|\nabla\delta w\|_{\dot B_{2,1}^{\frac{d}{2}-1}} \\\lesssim \|w_1-u_1\|_{\dot B_{2,1}^{\frac{d}{2}}}\|\delta w\|_{\dot B_{2,1}^{\frac{d}{2}}}$.
    \item[$\bullet$] $\|(Id-\mathbb{P})\left((u_1\cdot \nabla)\delta u\right)\|_{\dot B_{2,1}^{\frac{d}{2}-1}}\lesssim \|\left((u_1\cdot \nabla)\delta u\right)\|_{\dot B_{2,1}^{\frac{d}{2}-1}} \\ \lesssim \|u_1\|_{\dot B_{2,1}^{\frac{d}{2}-1}}\|\nabla \delta u\|_{\dot B_{2,1}^{\frac{d}{2}}}\lesssim\|u_1\|_{\dot B_{2,1}^{\frac{d}{2}-1}}\|\delta u\|_{\dot B_{2,1}^{\frac{d}{2}}}$.
\end{itemize}

By applying the localization operator $\dot\Delta_j$, applying the Cauchy-Schwarz inequality, by the lemma \ref{lemme edo}, multiplying by $2^{j(\frac{d}{2}-1)}$, summing over $j\in\Z$ and using product laws of Lemma \ref{Produit espace de Besov}, we deduce \eqref{unicité estimée u-v} with a specific treatment of the term $\dot\Delta_j(u_1\cdot \nabla)(\delta w-\delta u)$. Indeed we have: 
$$\displaylines{\int_{\R^d}\dot\Delta_j\left(u_1\cdot\nabla(\delta w-\delta u)\right)\cdot (\delta w_j-\delta u_j)dx \hfill\cr= \int_{\R^d}u_1\cdot\nabla(\delta w_j-\delta u_j)\cdot (\delta w_j-\delta u_j)dx \hfill\cr\hfill+\int_{\R^d}[\dot\Delta_j,u_1\cdot\nabla](\delta w-\delta u)\cdot (\delta w_j-\delta u_j)dx.}$$

By means of classical arguments, seen in the proof of inequality \eqref{unicité estimée v} (by integration by parts for the first term and using commutator estimates for the second), there exists a constant $C>0$ and a subsequence $(c_j)_{j\in\Z}$ such that $\sum_{j\in\Z}c_j\leq 1$ :
$$\displaylines{\left|\int_{\R^d}\dot\Delta_j\left(u_1\cdot\nabla(\delta w-\delta u)\right)\cdot (\delta w_j-\delta u_j)dx\right| \hfill\cr\hfill\leq C c_j 2^{-j(\frac{d}{2}-1)} \|u_1\|_{\dot B_{2,1}^{\frac{d}{2}+1}}\|\delta w-\delta u\|_{\dot B_{2,1}^{\frac{d}{2}-1}}\|\delta w_j-\delta u_j\|_{L^2}.}$$
\end{enumerate}
\end{proof}

Observing that for $i\in\{1,2\} $$$\displaylines{\|w_i-u_i\|_{\dot B_{2,1}^{\frac{d}{2}}}= \|w_i-u_i\|_{\dot B_{2,1}^{\frac{d}{2}}}^l+\|w_i-u_i\|_{\dot B_{2,1}^{\frac{d}{2}}}^h \hfill\cr\hfill \leq \|w_i-u_i\|_{\dot B_{2,1}^{\frac{d}{2}-1}}^l+\|w_i\|_{\dot B_{2,1}^{\frac{d}{2}+1}}^h+\|u_i\|_{\dot B_{2,1}^{\frac{d}{2}}},}$$
we deduce the following stability result by summing the inequalities of the previous lemma (multiplying by $\frac{1}{2 c_B^2}$) \eqref{unicité estimée u-v}) and absorbing the negligible terms of the right-hand member by the terms of the left-hand member:
\begin{prop}\label{proposition stabilité}
    Let $(\rho_1,w_1,u_1)$ and $(\rho_2,w_2,u_2)$ be two solutions of \eqref{Euler-Navier-Stokes3} with initial data $(\rho_{1,0},w_{1,0},u_{1,0})$ and $(\rho_{2,0},w_{2,0},u_{2,0})$ respectively and belonging to $E$. We set: $$\mathcal{Z}(t)\mathrel{\mathop:}= \|\delta \rho(t)\|_{\dot B_{2,1}^{\frac{d}{2}-1}}+\|\delta w(t)\|_{\dot B_{2,1}^{\frac{d}{2}}}+\|\delta u(t)\|_{\dot B_{2,1}^{\frac{d}{2}-1}}+\frac{1}{2c_B^2}\|(\delta w-\delta u)(t)\|_{\dot B_{2,1}^{\frac{d}{2}-1}}.$$  There exists a constant $c\in \R$ such that if for all $t\in[0,T]$, $\|u_1(t)\|_{\dot B_{2,1}^{\frac{d}{2}-1}}$ is smaller than $c$, then we have for any $t\in [0,T]$ :
    $$\displaylines{\mathcal{Z}(t)\leq \mathcal{Z}(0)+C\int_0^t \mathcal{Z}(\tau) \bigg(\|\rho_1\|_{\dot B_{2,1}^{\frac{d}{2}}}+\|(w_1,w_2)\|_{\dot B_{2,1}^{\frac{d}{2}+1}}+\|(u_1,u_2)\|_{\dot B_{2,1}^{\frac{d}{2}}\cap \dot B_{2,1}^{\frac{d}{2}+1}} \hfill\cr\hfill +\|(w_1-u_1,w_2-u_2)\|_{\dot B_{2,1}^{\frac{d}{2}-1}}\bigg) d\tau.}$$
\end{prop}

Let $(\rho_1,w_1,u_1)$ be the solution found in the existence part of Theorem \ref{théorème existence et unicité} and $(\rho_2,w_2,u_2)$ be a solution of \eqref{Euler-Navier-Stokes3} with the same regularity and initial data $(\rho_{0},w_{0},u_{0})$. 

By Proposition \ref{proposition stabilité} and Grönwall's Lemma, we obtain the uniqueness of the solution because $\mathcal{Z}(0)=0$.

\subsection{Existence of solutions}

Now let us prove the existence part of the theorem \ref{théorème existence et unicité}.

The first step is to approach the system \eqref{Euler-Navier-Stokes3}.

Let $J_n$ be the spectral truncation operator on $\displaystyle \{\xi\in\R^d,\;  n^{-1}\leq |\xi|\leq n\}$. 

Consider the following system: 
\begin{eqnarray*}
    \left\{\begin{array}{l}
      \partial_t \rho+\dive\left(J_n\left((J_n \rho) (J_n w)\right)\right)=0, \\
     \partial_t w + J_n\left(( J_n(w)\cdot \nabla)J_n w\right)+J_n w- J_n u=0, \\
     \partial_t u+J_n \mathbb{P}\left((J_n u\cdot \nabla) J_n u\right)=\Delta J_n u+J_n\mathbb{P}\left( J_n(\rho)(J_n w-J_n u)\right).
\end{array} \right.
\end{eqnarray*}

\begin{itemize}
    \item[$\bullet$] By the Cauchy-Lipschitz theorem, we have (using the spectral truncation operator) that the system admits a unique maximal solution $(\rho_n,w_n,u_n)\in \mathcal{C}^1([0,T^n[: L^2)$ with initial data ($J_n \rho_0, J_n w_0 ,J_n u_0)$ for all $n\in\N$.
\item[$\bullet$] We have $J_n \rho_n=\rho_n$, $J_n w_n=w_n$ and  $J_n u_n=u_n$ (using uniqueness in the previous system), whence 
\begin{eqnarray*}
    \left\{\begin{array}{l}
      \partial_t \rho_n+\dive\left(J_n\left(\rho_n w_n\right)\right)=0, \\
     \partial_t w_n + J_n\left((w_n\cdot \nabla)w_n\right)+w_n- u_n=0, \\
     \partial_t u_n+J_n\mathbb{P}\left((u_n\cdot \nabla) u_n\right)=\Delta u_n+J_n \mathbb{P}\left(\rho_n(w_n-u_n)\right).
\end{array} \right.
\end{eqnarray*}
\item[$\bullet$] From the a priori estimate \eqref{estimée théorème} and by the truncation and orthogonality properties of the operator $J_n$, we deduce (the index $n$ of the following terms corresponding to the sequence $(\rho_n,w_n,u_n)$) for all $t\in [0,T^n[$ : $$\mathcal{L}^n(t)+\int_0^t \mathcal{H}^n(\tau)d\tau \leq \mathcal{L}^n(0)\leq \mathcal{L}(0).$$
In particular, since $\mathcal{L}(0)$ and $\|\rho_0\|_{\dot B_{2,1}^{\frac{d}{2}}}$ are assumed to be small by hypothesis, by extension argument of the maximum solution, we have that $T^n=+\infty$. 
\end{itemize}

The previous estimates guarantee that $(\rho_n, w_n, u_n)_{n\in\N}$ is a bounded sequence in $\tilde{E}$ defined in \eqref{espace fonctionnel tilde E}.

In particular, we have for all $n\in\N$, $(w_n, u_n)$ bounded (by interpolation) in $L^2\left(\dot B_{2,1}^{\frac{d}{2}}\right)$ and $\rho_n$ bounded in $L_T^\infty(\dot B_{2,1}^{\frac{d}{2}})$.

Now, $\dot B_{2,1}^{\frac{d}{2}}$ is locally compact in $L^2$.

We can apply Ascoli's theorem and, with diagonal extraction, show that, up to extraction, the sequence of approximate solutions $(\rho_n, w_n,u_n)_{n\in\N}$ converges strongly to $(\rho,w,u)$ in $L^2([0,T[;L_{loc}^2(\R^3))$.

By classical arguments of weak compactness, we can conclude as in \cite{BCD} that $(\rho,w,u)$ belongs to $\tilde{E}$ and that $(\rho,w,u)$ is a solution of \eqref{Euler-Navier-Stokes2}.

\section{Time decay estimates}
This section is dedicated to proving Theorem \ref{estimées de décroissance} and \ref{estimées de décroissance2}.
\subsection{Time decay estimates for “critical” Besov spaces}
Let us prove the following precise result on decay estimates:
\begin{theorem}\label{estimées de décroissance précises}
    Under the same assumptions of Theorem \ref{théorème existence et unicité}, if the initial data $(\rho_0,w_0,u_0)$ also verifies \eqref{condition L1} then $w$ and $u$ satisfy the following inequality: \begin{equation}\label{estimée de décroissance besov}
        \|w(t)\|_{\dot B_{2,1}^{\frac{d}{2}+1}}+\|u(t)\|_{\dot B_{2,1}^{\frac{d}{2}-1}}+\|(w-u)(t)\|_{\dot B_{2,1}^{\frac{d}{2}-1}}\lesssim (1+c_0 t)^{\frac{1-d}{2}}\mathcal{Z}_0
    \end{equation}
where $c_0$ is a constant depending on $\|(w_0,u_0)\|_{\dot B_{2,\infty}^{-\frac{d}{2}}}$.

We have additional information for $d\geq 3$ : 
    \begin{equation}\label{estimée décroissance besov u-v 3d}
        \underset{t\geq 0}{\sup}\left((1+c_0 t)^{\frac{d+1}{2}}\|(w-u)(t)\|_{\dot B_{2,1}^{\frac{d}{2}-1}}\right)\lesssim \mathcal{Z}_0
    \end{equation}
 
In the case of dimension $d=2$, we have: 
    \begin{equation}\label{estimée décroissance besov u-v 2d}
        \underset{t\geq 0}{\sup}\left((1+c_0 t)\|(w-u)(t)\|_{\dot B_{2,1}^{0}}\right)\lesssim \mathcal{Z}_0
    \end{equation}
\end{theorem}

To prove this theorem, let us apply the method used by Nash in \cite{Nash}, as explained in the introduction.

\subsubsection{Propagation of negative regularity}
Let us first prove that regularity in $\dot B_{2,\infty}^{-\frac{d}{2}}$ propagates for any time with the following proposition:

\begin{prop}\label{propagation régularité négative}~

    Let $(\rho,w,u)$ be a solution of \eqref{Euler-Navier-Stokes2} with initial data~$(\rho_0,w_0,u_0)$ satisfying \eqref{condition initiale} and \eqref{condition L1}.

    Then, there holds for all $t\in\R^+$ : $$\underset{t\in\R^+}{\sup}\|(w,u)(t)\|_{\dot B_{2,\infty}^{-\frac{d}{2}}}\leq C \|(w_0,u_0)\|_{\dot B_{2,\infty}^{-\frac{d}{2}}}.$$
\end{prop}

\begin{proof} For this proof, we introduce the Chemin-Lerner space $\tilde{L}^q([0,T];\dot B_{2,r}^s)$ as the set of functions in $L^q([0,T];\dot B_{2,r}^s)$ with the norm \begin{equation}\label{norme tilde}\|u\|_{\tilde{L}_T^q(\dot B_{2,r}^s)}\mathrel{\mathop:}=\left\{\begin{array}{ll}
     \|(2^{js}\|u_j\|_{L_T^q(L^2)})_{j\in\Z}\|_{l^r(\Z)}<\infty & \text{if} \ 1\leq q < +\infty \\
      \|(2^{js}\underset{t\in [0,T]}{\sup}\|u_j(t)\|_{L^2})_{j\in\Z}\|_{l^r(Z)}<\infty & \text{if} \ q=\infty
\end{array}\right.\end{equation}

By Minkowski's inequality, we have: $$\|u\|_{\tilde{L}_T^q(\dot B_{2,r}^s)}\leq \|u\|_{L_T^q(\dot B_{2,r}^s)} \quad \text{if} \ r\geq q, \quad \|u\|_{\tilde{L}_T^q(\dot B_{2,r}^s)}\geq \|u\|_{L_T^q(\dot B_{2,r}^s)} \quad \text{if} \ r\leq q.$$
\begin{enumerate}
    \item By applying the localization operator $\dot\Delta_j$ to the second equation of \eqref{Euler-Navier-Stokes2}, we have: $$\partial_t w_j+w_j=u_j-\dot\Delta_j (w\cdot\nabla)w.$$

    Taking the scalar product in $L^2$ with $w_j$, by Lemma \ref{lemme edo}, by multiplying by $2^{-j\frac{d}{2}}$ and by taking the $\sup$ over $j\in\Z$, we obtain for any $t\in\R^+$ :
    $$\displaylines{\|w(t)\|_{\dot B_{2,\infty}^{-\frac{d}{2}}}+\|w\|_{\tilde{L}_t^1(\dot B_{2,\infty}^{-\frac{d}{2}})
    } \lesssim \|w_0\|_{\dot B_{2,\infty}^{-\frac{d}{2}}}+\|u\|_{\tilde{L}_t^1(\dot B_{2,\infty}^{-\frac{d}{2}})} \hfill\cr\hfill +\int_0^t\|(w\cdot \nabla)w\|_{\dot B_{2,\infty}^{-\frac{d}{2}}}d\tau.}$$

    Now, by product laws of Lemma \ref{Produit espace de Besov}, we have: $$\|\left((w\cdot \nabla) w\right)(t)\|_{\dot B_{2,\infty}^{-\frac{d}{2}}}\lesssim \big(\underset{\tau\in\R^+}{\sup}\|w(\tau)\|_{\dot B_{2,\infty}^{-\frac{d}{2}}}\big)\|w(t)\|_{\dot B_{2,1}^{\frac{d}{2}+1}}.$$

   We then have by a priori estimates \eqref{estimée théorème système 1} for all $t\in\R^+$ : 
    $$\displaylines{\|w(t)\|_{\dot B_{2,\infty}^{-\frac{d}{2}}}+\|w\|_{\tilde{L}_t^1(\dot B_{2,\infty}^{-\frac{d}{2}})
    }   \lesssim \|w_0\|_{\dot B_{2,\infty}^{-\frac{d}{2}}}+\|u\|_{\tilde{L}_t^1(\dot B_{2,\infty}^{-\frac{d}{2}})}+ \mathcal{Z}_0\big(\underset{\tau\in\R^+}{\sup}\|w(\tau)\|_{\dot B_{2,\infty}^{-\frac{d}{2}}}\big).}$$

    By smallness of $\mathcal{Z}_0$, we deduce: 
    \begin{equation}\label{régularité négative2}
        \|w\|_{L^\infty(\dot B_{2,\infty}^{-\frac{d}{2}})}+\|w\|_{\tilde{L}^1(\dot B_{2,\infty}^{-\frac{d}{2}})
    } \lesssim \|w_0\|_{\dot B_{2,\infty}^{-\frac{d}{2}}}+\|u\|_{\tilde{L}^1(\dot B_{2,\infty}^{-\frac{d}{2}})}.
    \end{equation}

    \item By applying the localization operator $\dot\Delta_j$ to the third equation of \eqref{Euler-Navier-Stokes2}, then taking the scalar product in $L^2$ with $u_j$, integrating by parts and by the Cauchy-Schwarz inequality, we obtain:

$$\frac{d}{dt}\|u_j\|_{L^2}^2+\|\nabla u_j\|_{L^2}^2\leq \left(\|\dot\Delta_j (u\cdot \nabla)u\|_{L^2}+\|\dot\Delta_j(\rho(w-u))\|_{L^2}\right)\|u_j\|_{L^2}.$$
    
    By Lemma \ref{lemme edo}, multiplying by $2^{-j\frac{d}{2}}$ and taking the $\sup$ over the $j\in\Z$, we obtain for all $t\in\R^+$ :
    $$\displaylines{\|u(t)\|_{\dot B_{2,\infty}^{-\frac{d}{2}}}+\|u\|_{\tilde{L}_t^1(\dot B_{2,\infty}^{2-\frac{d}{2}})}d\tau \lesssim \|u_0\|_{\dot B_{2,\infty}^{-\frac{d}{2}}}+\|(u\cdot \nabla)u\|_{\tilde{L}_t^1(\dot B_{2,\infty}^{-\frac{d}{2}})} \cr\hfill +\|\rho(w-u)\|_{\tilde{L}_t^1(\dot B_{2,\infty}^{-\frac{d}{2}})}.}$$

    Now, by product laws of Lemma \ref{Produit espace de Besov} and by \eqref{estimée théorème système 1}, we have for all $t\in\R^+$: $$\|(u\cdot \nabla)u\|_{\tilde{L}_t^1(\dot B_{2,\infty}^{-\frac{d}{2}})}\lesssim \|u\|_{\tilde{L}_t^\infty(\dot B_{2,\infty}^{-\frac{d}{2}})}\|u\|_{\tilde{L}_t^1(\dot B_{2,1}^{\frac{d}{2}+1})}\lesssim \mathcal{Z}(0)\|u\|_{\tilde{L}_t^\infty(\dot B_{2,\infty}^{-\frac{d}{2}})},$$
and $$\|\rho (w-u)\|_{L_t^1(\dot B_{2,\infty}^{-\frac{d}{2}})}\lesssim \|\rho\|_{\tilde{L}_t^\infty(\dot B_{2,1}^{\frac{d}{2}})}\|w-u\|_{\tilde{L}_t^1(\dot B_{2,\infty}^{-\frac{d}{2}})} \lesssim \mathcal{Z}(0) \ \|w-u\|_{\tilde{L}_t^1(\dot B_{2,\infty}^{-\frac{d}{2}})}.$$

    By smallness of $\mathcal{Z}_0$, we deduce: 
    \begin{multline}\label{régularité négative3}
        \|u\|_{\tilde{L}_t^\infty(\dot B_{2,\infty}^{-\frac{d}{2}})}+\|u\|_{\tilde{L}_t^1(\dot B_{2,\infty}^{-\frac{d}{2}})}d\tau\lesssim \|u_0\|_{\dot B_{2,\infty}^{-\frac{d}{2}}}+\mathcal{Z}_0 \ \|w-u\|_{\tilde{L}_t^1(\dot B_{2,\infty}^{-\frac{d}{2}})}.
    \end{multline}

\item Applying the localization operator $\dot \Delta_j$ to the equation verified by $w-u$, we obtain:
$$\displaylines{\partial_t (w_j-u_j)+w_j-u_j=-\dot\Delta_j (w\cdot \nabla)w+\dot\Delta_j(u\cdot \nabla)u-\Delta u_j -\dot\Delta_j (\rho(w-u)).}$$

Taking the scalar product in $L^2$ with $w_j-u_j$, by Lemma \ref{lemme edo}, then multiplying by $2^{-j\frac{d}{2}}$ and taking the $sup$ over the $j\in\Z$ and by the product laws of Lemma \ref{Produit espace de Besov}, we obtain for any $t\in\R^+$ :
$$\displaylines{\|(w-u)(t)\|_{\dot B_{2,\infty}^{-\frac{d}{2}}}+\|w-u\|_{\tilde{L}_t^1(\dot B_{2,\infty}^{-\frac{d}{2}})} \hfill\cr \lesssim \|w_0-u_0\|_{\dot B_{2,\infty}^{-\frac{d}{2}}}+\|w\|_{\tilde{L}_t^\infty(\dot B_{2,\infty}^{-\frac{d}{2}})}\|w\|_{\tilde{L}_t^1(\dot B_{2,1}^{\frac{d}{2}+1})}+\|u\|_{\tilde{L}_t^\infty(\dot B_{2,\infty}^{-\frac{d}{2}})}\|u\|_{\tilde{L}_t^1(\dot B_{2,1}^{\frac{d}{2}+1})} \hfill\cr\hfill+\|u\|_{\tilde{L}_t^1(\dot B_{2,1}^{\frac{d}{2}+1})} +\|\rho\|_{\tilde{L}_t^\infty(\dot B_{2,1}^{\frac{d}{2}})}\|w-u\|_{\tilde{L}_t^1(\dot B_{2,\infty}^{-\frac{d}{2}})}.}$$

By \eqref{estimée théorème système 1} and  smallness of $\mathcal{Z}_0$, we then have the following a priori estimate for all $t\in\R^+$: 
 \begin{multline}\label{régularité négative4}
\|(w-u)(t)\|_{\dot B_{2,\infty}^{-\frac{d}{2}}}+\|w-u\|_{\tilde{L}_t^1(\dot B_{2,\infty}^{-\frac{d}{2}})} \\ \lesssim \|(w_0,u_0)\|_{\dot B_{2,\infty}^{-\frac{d}{2}}}+\mathcal{Z}(0)\|w\|_{\tilde{L}_t^\infty(\dot B_{2,\infty}^{-\frac{d}{2}})}  +\mathcal{Z}(0)\|u\|_{\tilde{L}_t^\infty(\dot B_{2,\infty}^{-\frac{d}{2}})} \\ +\|u\|_{\tilde{L}_t^1(\dot B_{2,1}^{\frac{d}{2}+1})}.
 \end{multline}
    \end{enumerate}

    By summing up inequalities \eqref{régularité négative2}, \eqref{régularité négative3} and \eqref{régularité négative4}
    and taking account of the smallness of $\mathcal{Z}_0$, we deduce: $$\underset{t\in\R^+}{\sup}\left(\|(w,u)(t)\|_{\dot B_{2,\infty}^{-\frac{d}{2}}}\right)\lesssim \|(w_0,u_0)\|_{\dot B_{2,\infty}^{-\frac{d}{2}}}.$$
    Whence the Proposition.
\end{proof}

\subsubsection{Decay estimates on velocities}~\\
Being monotonous, the functional $\mathcal{L}$ defined in \eqref{définition de L} is almost everywhere differentiable on $\R^+$ and the inequality \eqref{inégalité fonctionnelle} implies thus: \begin{equation}\label{inégalité différentielle}\frac{d}{dt}\mathcal{L}+\mathcal{H}\leq 0 \quad \text{a.e.  on} \ \R^+. \end{equation}

By interpolation inequality, we have: $$\|u\|_{\dot B_{2,1}^{\frac{d}{2}-1}}^l \lesssim \left(\|u\|_{\dot B_{2,1}^{\frac{d}{2}+1}}^l\right)^{1-\theta_0}\left(\|u\|_{\dot B_{2,\infty}^{-\frac{d}{2}}}^l\right)^{\theta_0} \quad \text{with} \ \theta_0=\frac{2}{1+d}.$$

So we have by Proposition \ref{propagation régularité négative} : 
$$\|u(t)\|_{\dot B_{2,1}^{\frac{d}{2}+1}}^l\gtrsim \left(\|u(t)\|_{\dot B_{2,1}^{\frac{d}{2}-1}}^l\right)^{\frac{1}{1-\theta_0}}\left(\|(w_0,u_0)\|_{\dot B_{2,\infty}^{-\frac{d}{2}}}\right)^{-\frac{\theta_0}{1-\theta_0}}.$$

By a priori estimates \eqref{estimée théorème} and condition \eqref{condition initiale}, we have: $$\displaylines{\|u(t)\|_{\dot B_{2,1}^{\frac{d}{2}-1}}^h  +\|w(t)\|_{\dot B_{2,1}^{\frac{d}{2}+1}}+\|(w-u)(t)\|_{\dot B_{2,1}^{\frac{d}{2}-1}} \hfill\cr\hfill \leq  \|u(t)\|_{\dot B_{2,1}^{\frac{d}{2}+1}}^h+\|w(t)\|_{\dot B_{2,1}^{\frac{d}{2}+1}}+\|(w-u)(t)\|_{\dot B_{2,1}^{\frac{d}{2}-1}} \cr\hfill \lesssim 1.}$$
Since also $\frac{1}{1-\theta_0}\geq 1$, we deduce that:
$$\displaylines{\|u(t)\|_{\dot B_{2,1}^{\frac{d}{2}+1}}^h+\|w(t)\|_{\dot B_{2,1}^{\frac{d}{2}+1}}+\|(w-u)(t)\|_{\dot B_{2,1}^{\frac{d}{2}-1}}\hfill\cr\hfill\geq \left(\|u(t)\|_{\dot B_{2,1}^{\frac{d}{2}-1}}^h+\|w(t)\|_{\dot B_{2,1}^{\frac{d}{2}+1}}+\|(w-u)(t)\|_{\dot B_{2,1}^{\frac{d}{2}-1}}\right)^{\frac{1}{1-\theta_0}}} $$

Thus, there exists $c>0$ small enough such that: \begin{equation}\label{inégalité H}\mathcal{H}\geq c_0 \mathcal{L}^{1-\theta_0} \quad \text{with} \ c_0\mathrel{\mathop:}=c\left(\|(w_0,u_0)\|_{\dot B_{2,\infty}^{-\frac{d}{2}}}+1\right)^{-\frac{\theta_0}{1-\theta_0}}.\end{equation}

By the inequality \eqref{inégalité fonctionnelle} and the previous inequality, we obtain: $$\frac{d}{dt}\mathcal{L}+c_0 \mathcal{L}^{\frac{1}{1-\theta_0}}\leq 0.$$

Whence the decay estimate \eqref{estimée de décroissance besov}.

In order to obtain \eqref{taux de décroissance L2}, we interpolate between $\dot B_{2,1}^{\frac{d}{2}-1}$ and $\dot B_{2,\infty}^{-\frac{d}{2}}$ : $$\|u\|_{L^2}\simeq \|u\|_{\dot B_{2,2}^0}\leq \|u\|_{\dot B_{2,1}^{\frac{d}{2}-1}}^{1-\theta}\|u\|_{B_{2,\infty}^{-\frac{d}{2}}}^\theta, \quad  \text{with} \ \theta=\frac{d-2}{2(d-1)}.$$

We then have owing to \eqref{estimée de décroissance besov} and Proposition \ref{propagation régularité négative}, $$\|u(t)\|_{L^2}\lesssim \|(w_0,u_0)\|_{\dot B_{2,\infty}^{-\frac{d}{2}}}^\theta \mathcal{Z}_0^{1-\theta}(1+c_0 t)^{-\frac{d}{4}}.$$

The same applies to $w$. This gives \eqref{taux de décroissance L2}.

\subsubsection{Improvement of estimates on the difference of velocities}~\\
In the case of $d\geq 3$, let us prove \eqref{estimée décroissance besov u-v 3d}, we can then interpolate, as before, to get \eqref{taux de décroissance L2 mode amorti}. 

Let us distinguish between high and low frequency behavior.

\begin{lemma}\label{estimée décroissance u,v hf}
    We have the following inequality for $d\geq 3$, for all $t\in\R^+$ :
    $$\displaylines{\underset{t\geq 0}{\sup}\left((1+c_0 t)^{\frac{d+1}{2}}\|(w,u)(t)\|_{\dot B_{2,1}^{\frac{d}{2}-1}}^h\right) \hfill\cr \lesssim \mathcal{Z}(0)+\mathcal{Z}(0) \ \underset{t\geq 0}{\sup}\left((1+c_0 t)^{\frac{d+1}{2}}\|(w-u)(t)\|_{\dot B_{2,1}^{\frac{d}{2}-1}}^l\right).}$$
\end{lemma}

\begin{proof}
First, let us look at the equation on $u$ :
$$\partial_t u-\Delta u=-\mathbb{P}(u\cdot \nabla)u+\mathbb{P}\rho (w-u).$$

By applying the operator $\dot\Delta_j$ for $j\geq 0$ and applying the scalar product with $u_j$, we obtain:
$$\displaylines{\frac{1}{2}\frac{d}{dt}\|u_j\|_{L^2}+\|\nabla u_j\|_{L^2}^2=-\int_{\R^d}\dot\Delta_j\mathbb{P}\left((u-w)\cdot \nabla)u\right)\cdot u_j dx \hfill\cr\hfill -\int_{\R^d}\dot\Delta_j\mathbb{P}(w\cdot\nabla)u\cdot u_j dx  +\int_{\R^d}\dot\Delta_j\mathbb{P}\rho (w-u)\cdot v_j dx.}$$

Now, by continuity of $\mathbb{P}$, by the Cauchy-Schwarz inequality and by the product laws of Lemma \ref{Produit espace de Besov}, we have the existence of a sequence $(c_j)_{j\geq 0}$ such that $\sum_{j\geq 0}c_j=1$ satisfying : \begin{align*} \bigg|  - \int_{\R^d}\dot\Delta_j\mathbb{P}\left((u-w)\cdot \nabla)u\right)\cdot u_j & dx\bigg| \\ & \leq \|\dot\Delta_j \left((u-w)\cdot \nabla\right) u\|_{L^2} \|u_j\|_{L^2}
\\ & \leq  C c_j 2^{-j\left(\frac{d}{2}-1\right)}\|w-u\|_{\dot B_{2,1}^{\frac{d}{2}-1}}\|u\|_{\dot B_{2,1}^{\frac{d}{2}+1}}\|u_j\|_{L^2}. \end{align*}

Furthermore, by symmetry of $\mathbb{P}$ and integration by parts, we obtain: 
\begin{align*}
    \int_{\R^d}\dot\Delta_j\mathbb{P}\left((w\cdot\nabla)u\right)\cdot u_j dx & = \int_{\R^d}(w\cdot\nabla)u_j\cdot u_j dx+\int_{\R^d}[\dot\Delta_j,w\cdot\nabla]u\cdot u_j dx 
    \\ & =  -\frac{1}{2} \int_{\R^d}(\dive w)|u_j|^2 dx+\int_{\R^d}[\dot\Delta_j,w\cdot\nabla]u\cdot u_j dx. 
\end{align*}

By Hölder's inequality and the commutator estimates of Lemma \ref{commutateur}, there exists a sequence $(c_j)_{j\geq 0}$ such that $\sum_{j\geq 0}c_j=1$ and
$$\displaylines{\left|\int_{\R^d}\dot\Delta_j \mathbb{P}\left((w\cdot\nabla)u\right)\cdot u_j dx\right| \leq C c_j 2^{-j\left(\frac{d}{2}-1\right)} \|w\|_{\dot B_{2,1}^{\frac{d}{2}+1}}\|u\|_{\dot B_{2,1}^{\frac{d}{2}-1}}\|u_j\|_{L^2} .}$$

We have also: $$\left|\int_{\R^d}\mathbb{P}\rho(w-u)\cdot u_j dx\right| \leq C c_j 2^{-j\left(\frac{d}{2}-1\right)}\|\rho\|_{\dot B_{2,1}^{\frac{d}{2}}}\|w-u\|_{\dot B_{2,1}^{\frac{d}{2}-1}}\|u_j\|_{L^2}.$$

We deduce the following inequality: 
$$\displaylines{\frac{1}{2}\frac{d}{dt}\|u_j\|_{L^2}^2+2^{2j}\|u_j\|_{L^2}^2\lesssim c_j 2^{-j\left(\frac{d}{2}-1\right)}\bigg(\|w-u\|_{\dot B_{2,1}^{\frac{d}{2}-1}}\|u\|_{\dot B_{2,1}^{\frac{d}{2}+1}} \cr\hfill +\|w\|_{\dot B_{2,1}^{\frac{d}{2}+1}}\|u\|_{\dot B_{2,1}^{\frac{d}{2}-1}} +\|\rho\|_{\dot B_{2,1}^{\frac{d}{2}}}\|w-u\|_{\dot B_{2,1}^{\frac{d}{2}-1}}\bigg)\|u_j\|_{L^2}.}$$

We then have for all $j\geq 0$ : $$\displaylines{\|u_j(t)\|_{L^2}\leq C e^{-c_0 2^{2j}t}\|u_{j,0}\|_{L^2}  \hfill\cr\hfill+C c_j 2^{-j\left(\frac{d}{2}-1\right)} \int_0^t e^{-c_0 2^{2j}(t-\tau)}\bigg(\|w-u\|_{\dot B_{2,1}^{\frac{d}{2}-1}}\|u\|_{\dot B_{2,1}^{\frac{d}{2}+1}} \cr\hfill +\|w\|_{\dot B_{2,1}^{\frac{d}{2}+1}}\|u\|_{\dot B_{2,1}^{\frac{d}{2}-1}} +\|\rho\|_{\dot B_{2,1}^{\frac{d}{2}}}\|w-u\|_{\dot B_{2,1}^{\frac{d}{2}-1}}\bigg)d\tau.}$$

Noting that $e^{-c_0 2^{2j}t}\leq e^{-c_0 t}$ for $j\geq 0$, multiplying by $\displaystyle 2^{j\left(\frac{d}{2}-1\right)}$ and summing over the $j\geq 0$, we obtain by multiplying by $(1+c_0 t)^{\frac{d}{2}+1}$: 
$$\displaylines{(1+c_0 t)^{\frac{d}{2}+1}\|u(t)\|_{\dot B_{2,1}^{\frac{d}{2}-1}}^h\leq C (1+c_0 t)^{\frac{d}{2}+1} e^{-c_0 t}\|u_{0}\|_{\dot B_{2,1}^{\frac{d}{2}-1}} \hfill\cr\hfill +C \int_0^t (1+c_0 t)^{\frac{d}{2}+1} e^{-c_0(t-\tau)}\bigg(\|w-u\|_{\dot B_{2,1}^{\frac{d}{2}-1}}\|u\|_{\dot B_{2,1}^{\frac{d}{2}+1}} \cr\hfill+\|w\|_{\dot B_{2,1}^{\frac{d}{2}+1}}\|u\|_{\dot B_{2,1}^{\frac{d}{2}-1}} +\|\rho\|_{\dot B_{2,1}^{\frac{d}{2}}}\|w-u\|_{\dot B_{2,1}^{\frac{d}{2}-1}}\bigg)d\tau.}$$

In order to estimate the terms in the right-hand side of the previous inequality, one can split the integral into two parts as follows:
$$\displaylines{(1+c_0 t)^{\frac{d+1}{2}}\int_0^t e^{-c_0(t-\tau)}\|w-u\|_{\dot B_{2,1}^{\frac{d}{2}-1}}\|u\|_{\dot B_{2,1}^{\frac{d}{2}+1}}d\tau
\hfill\cr\hfill  =(1+c_0 t)^{\frac{d+1}{2}}\int_0^{t/2}e^{-c_0(t-\tau)}\|w-u\|_{\dot B_{2,1}^{\frac{d}{2}-1}}\|u\|_{\dot B_{2,1}^{\frac{d}{2}+1}}d\tau \hfill\cr\hfill+\int_{t/2}^t \left(\frac{1+c_0 t}{1+c_0 \tau}\right)^{\frac{d+1}{2}}e^{-c_0(t-\tau)}(1+c_0 \tau)^{\frac{d+1}{2}}\|w-u\|_{\dot B_{2,1}^{\frac{d}{2}-1}}\|u\|_{\dot B_{2,1}^{\frac{d}{2}+1}}d\tau.}$$

By inequality \eqref{estimée théorème système 1}, we deduce: 
$$\displaylines{(1+c_0 t)^{\frac{d+1}{2}}\int_0^{t/2}e^{-c_0(t-\tau)}\|w-u\|_{\dot B_{2,1}^{\frac{d}{2}-1}}\|u\|_{\dot B_{2,1}^{\frac{d}{2}+1}}d\tau \hfill\cr\hfill \lesssim(1+c_0 t)^{\frac{d+1}{2}}e^{-c_0 t/2}\int_0^{t/2} \|w-u\|_{\dot B_{2,1}^{\frac{d}{2}-1}}\|u\|_{\dot B_{2,1}^{\frac{d}{2}+1}}d\tau \cr\hfill  \lesssim\left(\underset{t\geq 0}{\sup}\|(w-u)(t)\|_{\dot B_{2,1}^{\frac{d}{2}-1}}\right)\int_0^t \|u\|_{\dot B_{2,1}^{\frac{d}{2}+1}}d\tau\cr\hfill \lesssim \mathcal{Z}(0))^2,}$$ and $$\displaylines{ \int_{\frac{t}{2}}^t e^{-c_0(t-\tau)}(1+c_0 \tau)^{\frac{d+1}{2}}\|w-u\|_{\dot B_{2,1}^{\frac{d}{2}-1}}\|u\|_{\dot B_{2,1}^{\frac{d}{2}+1}}d\tau
\hfill\cr \lesssim\left(\underset{t\geq 0}{\sup}(1+c_0 t)^{\frac{d+1}{2}}\|w-u\|_{\dot B_{2,1}^{\frac{d}{2}-1}}\right)\int_{t/2}^t \|u\|_{\dot B_{2,1}^{\frac{d}{2}+1}}d\tau
\hfill\cr \lesssim\mathcal{Z}(0)\left(\underset{t\geq 0}{\sup}(1+c_0 t)^{\frac{d+1}{2}}\|w-u\|_{\dot B_{2,1}^{\frac{d}{2}-1}}\right).\hfill
}$$

We have also by Lemma \ref{int exp}:
$$\displaylines{(1+c_0 t)^{\frac{d+1}{2}}\int_0^t e^{-c_0 (t-\tau)}\|\rho\|_{\dot B_{2,1}^{\frac{d}{2}}}\|w-u\|_{\dot B_{2,1}^{\frac{d}{2}-1}} d\tau
\hfill\cr
=\int_0^t e^{-c_0(t-\tau)}\left(\frac{1+c_0 t}{1+c_0 \tau}\right)^{\frac{d+1}{2}}(1+c_0 \tau)^{\frac{d+1}{2}}\|\rho\|_{\dot B_{2,1}^{\frac{d}{2}}}\|w-u\|_{\dot B_{2,1}^{\frac{d}{2}-1}}d\tau
\hfill\cr \lesssim \left(\underset{t\geq 0}{\sup}(1+c_0 t)^{\frac{d+1}{2}}\|w-u\|_{\dot B_{2,1}^{\frac{d}{2}-1}}\right) \left(\underset{t\geq 0}{\sup}\|\rho(t)\|_{\dot B_{2,1}^{\frac{d}{2}}}\right)
\hfill\cr ~~~~~\times\int_0^t e^{-c_0(t-\tau)}\left(\frac{1+c_0 t}{1+c_0 \tau}\right)^{\frac{d+1}{2}}d\tau \hfill\cr \lesssim\mathcal{Z}(0)\left(\underset{t\geq 0}{\sup}(1+c_0 t)^{\frac{d+1}{2}}\|w-u\|_{\dot B_{2,1}^{\frac{d}{2}-1}}\right).\hfill } $$

If we take $$A(t)\mathrel{\mathop:=(1+c_0 t)^{\frac{d+1}{2}}\int_0^t e^{-c_0(t-\tau)}\|u\|_{\dot B_{2,1}^{\frac{d}{2}-1}}\|w\|_{\dot B_{2,1}^{\frac{d}{2}+1}} d\tau},$$ we have also
$$\displaylines{A(t)\lesssim \int_0^t e^{-c_0 (t-\tau)}\left(\frac{1+c_0 t}{1+c_0 \tau}\right)^{\frac{d+1}{2}}(1+c_0 \tau)^{\frac{d-1}{2}} \|u\|_{\dot B_{2,1}^{\frac{d}{2}-1}}(1+c_0 \tau)\|w\|_{\dot B_{2,1}^{\frac{d}{2}+1}} d\tau,}$$
and we can deduce from the inequality \eqref{estimée de décroissance besov} (since $d\geq 3$) : 

\begin{multline}\label{cas 1} A(t) \lesssim \underset{t\geq 0}{\sup}\left((1+c_0 t)^{\frac{d-1}{2}}\|u(t)\|_{\dot B_{2,1}^{\frac{d}{2}-1}}\right) \\ ~~~~~~~~~~~~~~~~~~~~~~~\times\underset{t\geq 0}{\sup}\left((1+c_0 t)^{\frac{d-1}{2}}\|w(t)\|_{\dot B_{2,1}^{\frac{d}{2}+1}}\right) \int_0^t e^{-c_0(t-\tau)}\left(\frac{1+c_0 t}{1+c_0 \tau}\right)^{\frac{d+1}{2}}d\tau
\\  \lesssim \mathcal{Z}(0).
\end{multline}

We can therefore deduce: 
\begin{multline}\label{estimée décroissance v hf}(1+c_0 t)^{\frac{d}{2}+1}\|u(t)\|_{\dot B_{2,1}^{\frac{d}{2}-1}}^h\lesssim \mathcal{Z}(0) \\ +\mathcal{Z}(0)\left(\underset{t\geq 0}{\sup}(1+c_0 t)^{\frac{d+1}{2}}\|w-u\|_{\dot B_{2,1}^{\frac{d}{2}-1}}\right).\end{multline}

By applying the localisation operator $\dot\Delta_j$ to the equation on $w$ of system \eqref{Euler-Navier-Stokes2}, we have by Duhamel's formula: 
$$w_j(t)=e^{-t}w_{j,0}+\int_0^t e^{-(t-\tau)}\left(u_j(\tau)-\dot\Delta_j(w\cdot\nabla)w(\tau)\right)d\tau.$$

By setting $$B(t)\mathrel{\mathop:}=(1+c_0 t)^{\frac{d+1}{2}}\|w(t)\|_{\dot B_{2,1}^{\frac{d}{2}-1}}^h,$$
we have for $d\geq 3$:
$$\displaylines{B(t) \lesssim \|w_0\|_{\dot B_{2,1}^{\frac{d}{2}-1}}^h +\int_0^t \left(\frac{1+c_0 t}{1+c_0 \tau}\right)^{\frac{d+1}{2}}e^{-(t-\tau)}(1+c_0 \tau)^{\frac{d+1}{2}}\|u(\tau)\|_{\dot B_{2,1}^{\frac{d}{2}-1}}^hd\tau \hfill\cr\hfill +\int_0^t \left(\frac{1+c_0 t}{1+c_0 \tau}\right)^{\frac{d+1}{2}}e^{-(t-\tau)}(1+c_0 \tau)^{\frac{d-1}{2}}(1+c_0 \tau)\|(w\cdot\nabla)w\|_{\dot B_{2,1}^{\frac{d}{2}-1}}d\tau.}$$

By Lemma \ref{int exp} and inequality \eqref{estimée théorème système 1}, we have
$$\displaylines{B(t)  \lesssim \|w_0\|_{\dot B_{2,1}^{\frac{d}{2}-1}}^h+\left(\underset{t\geq 0}{\sup}(1+c_0 t)^{\frac{d+1}{2}}\|u(\tau)\|_{\dot B_{2,1}^{\frac{d}{2}-1}}^h\right) \hfill 
\cr + \left(\underset{t\geq 0}{\sup}(1+c_0 t)^{\frac{d-1}{2}}\|w\|_{\dot B_{2,1}^{\frac{d}{2}-1}}\right)\left(\underset{t\geq 0}{\sup}(1+c_0 t)^{\frac{d-1}{2}}\|w\|_{\dot B_{2,1}^{\frac{d}{2}+1}}\right) \cr\hfill \times \int_0^t \left(\frac{1+c_0 t}{1+c_0 \tau}\right)^{\frac{d+1}{2}} e^{-(t-\tau)} d\tau
\cr\hfill\lesssim\mathcal{Z}(0)+\left(\underset{t\geq 0}{\sup}(1+c_0 t)^{\frac{d+1}{2}}\|u(\tau)\|_{\dot B_{2,1}^{\frac{d}{2}-1}}^h\right).\hfill
}$$
So we have the inequality:
\begin{equation}\label{estimée décroissance u hf}
    \underset{t\geq 0}{\sup}\left((1+c_0 t)^{\frac{d+1}{2}}\|w(t)\|_{\dot B_{2,1}^{\frac{d}{2}-1}}^h\right)\lesssim\mathcal{Z}(0)+\left(\underset{t\geq 0}{\sup}(1+c_0 t)^{\frac{d+1}{2}}\|u(\tau)\|_{\dot B_{2,1}^{\frac{d}{2}-1}}^h\right).
\end{equation}

By combining inequalities \eqref{estimée décroissance v hf} and \eqref{estimée décroissance u hf}, we obtain the result of the lemma.
\end{proof}

For low frequencies, we have the following lemma:
\begin{lemma}\label{lemme dimension 3}
    We have the following estimate at low frequencies for $d\geq 3$ : 
    \begin{multline}\label{estimée décroissance u-v bf 3d}\underset{t\geq 0}{\sup}\left((1+c_0 t)^{\frac{d+1}{2}}\|w-u\|_{\dot B_{2,1}^{\frac{d}{2}-1}}^l\right)\lesssim \mathcal{Z}_0 \\  + \mathcal{Z}_0 \ \underset{t\geq 0}{\sup}\left((1+c_0 t)^{\frac{d+1}{2}}\|w-u\|_{\dot B_{2,1}^{\frac{d}{2}-1}}^h\right).\end{multline}
\end{lemma}

\begin{proof}
Let us consider the equations on $w-u$ and $u$:
$$\left\{\begin{array}{l} \partial_t (w-u)+w-u+\Delta u=F_1, \\ \partial_t u-\Delta u=F_2, \end{array}\right.$$ where $$F\mathrel{\mathop:}=\begin{pmatrix}
    F_1 \\ F_2
\end{pmatrix}\mathrel{\mathop:}=\begin{pmatrix}
   \mathbb{P}(u\cdot\nabla)u-(w\cdot\nabla)w-\mathbb{P}\rho(w-u) \\ -\mathbb{P}(u\cdot\nabla)u+\mathbb{P}\rho(w-u)
\end{pmatrix}.$$

By Fourier transform, we obtain: $$\frac{d}{dt}\begin{pmatrix}
    \widehat{w-u} \\ \widehat{u} 
\end{pmatrix}+\begin{pmatrix}
    1 & -|\xi|^2 \\ 0 & |\xi|^2
\end{pmatrix}\begin{pmatrix}
    \widehat{w-u} \\ \widehat{u}
\end{pmatrix}=\widehat{F}.$$

With the localisation operator $\dot\Delta_j$, the fact that $|\xi|^2\simeq 2^{2j}$ and Duhamel's formula, we obtain by Lemma \ref{semi-groupe} : 
$$\displaylines{\|(w-u)_j(t)\|_{L^2}\lesssim e^{-t}\|(w-u)_{0,j}\|_{L^2}+te^{-t 2^{2j}}2^{2j}\|u_{0,j}\|_{L^2} \hfill\cr\hfill+\int_0^t e^{-(t-\tau)}\|\dot\Delta_j\left(\mathbb{P}(u\cdot\nabla)u-(w\cdot\nabla)w-\mathbb{P}\rho(w-u)\right)\|_{L^2}d\tau \hfill\cr\hfill+\int_0^t (t-\tau)e^{-(t-\tau)2^{2j}}2^{2j}\|\dot\Delta_j\left(-\mathbb{P}(u\cdot\nabla)
u+\mathbb{P}\rho(w-u)\right)\|_{L^2} d\tau.} $$

We then have by multiplying by $(1+c_0t)^{\frac{d+1}{2}} 2^{j(\frac{d}{2}-1)}$, summing over the $j\in\Z^{-}$ and by the inequality $xe^{-x}\leq C e^{-cx}$ for all $x\in\R_+^*$ : $$\displaylines{(1+c_0 t)^{\frac{d+1}{2}}\|w-u\|_{\dot B_{2,1}^{\frac{d}{2}-1}}^l \hfill\cr \lesssim (1+c_0 t)^{\frac{d+1}{2}}e^{-t}\mathcal{Z}(0)+\int_0^t (1+c_0 t)^{\frac{d+1}{2}}e^{-(t-\tau)}\|\mathbb{P}(u\cdot\nabla)u-(w\cdot\nabla)w\|_{\dot B_{2,1}^{\frac{d}{2}-1}}^l d\tau \hfill\cr\hfill+\int_0^t (1+c_0 t)^{\frac{d+1}{2}}e^{-(t-\tau)}\|\mathbb{P}\rho(w-u)\|_{\dot B_{2,1}^{\frac{d}{2}-1}}^l d\tau \cr\hfill+\int_0^t (1+c_0 t)^{\frac{d+1}{2}}e^{-c(t-\tau)}\|\mathbb{P}(u\cdot\nabla)u\|_{\dot B_{2,1}^{\frac{d}{2}-1}}^l d\tau.}$$

Let us estimate the terms on the right one by one.
\begin{itemize}
    \item[$\bullet$] To begin with, we have: $$(1+c_0 t)^{\frac{d+1}{2}}e^{-t}\mathcal{Z}(0)\lesssim \mathcal{Z}(0).$$
    \item[$\bullet$] By product laws and \eqref{estimée théorème système 1}, we have: $$\|\mathbb{P}\rho(w-u)\|_{\dot B_{2,1}^{\frac{d}{2}-1}}\lesssim \|\rho\|_{\dot B_{2,1}^{\frac{d}{2}}}\|w-u\|_{\dot B_{2,1}^{\frac{d}{2}-1}}\lesssim \mathcal{Z}(0)\|w-u\|_{\dot B_{2,1}^{\frac{d}{2}-1}}.$$ 
So we have: $$\displaylines{\int_0^t (1+c_0 t)^{\frac{d+1}{2}}e^{-(t-\tau)}\|\mathbb{P}\rho(w-u)\|_{\dot B_{2,1}^{\frac{d}{2}-1}}^l d\tau \hfill\cr\lesssim \mathcal{Z}_0 \ \underset{t\geq 0}{\sup}\left((1+c_0 t)^{\frac{d+1}{2}}\|(w-u)(t)\|_{\dot B_{2,1}^{\frac{d}{2}-1}}\right)\int_0^t \left(\frac{1+c_0 t}{1+c_0 \tau}\right)^{\frac{d+1}{2}}e^{-(t-\tau)} d\tau\hfill\cr \lesssim\mathcal{Z}_0 \ \underset{t\geq 0}{\sup}\left((1+c_0 t)^{\frac{d+1}{2}}\|(w-u)(t)\|_{\dot B_{2,1}^{\frac{d}{2}-1}}\right). \hfill }$$
    \item[$\bullet$] We have the following relation: $$\displaylines{\mathbb{P}(u\cdot\nabla)u-(w\cdot\nabla)w=\mathbb{P}\left((u-w)\cdot\nabla\right)u+\mathbb{P}(w\cdot\nabla)(u-w)\hfill\cr\hfill(\mathbb{P}-Id)(w\cdot \nabla)w.}$$

    By the product laws of Lemma \ref{Produit espace de Besov}, we have for $d\geq 3$: \begin{equation}\label{cas 2}\begin{aligned}  \|\mathbb{P} & (u\cdot\nabla)u-(w\cdot\nabla)w\|_{\dot B_{2,1}^{\frac{d}{2}-1}}^l \\  \lesssim & \|w-u\|_{\dot B_{2,1}^{\frac{d}{2}-1}}\|u\|_{\dot B_{2,1}^{\frac{d}{2}+1}}+\|(w\cdot\nabla)(u-w)\|_{\dot B_{2,1}^{\frac{d}{2}-2}}^l+\|w\|_{\dot B_{2,1}^{\frac{d}{2}-1}}\|w\|_{\dot B_{2,1}^{\frac{d}{2}+1}}\\
    \lesssim & \|w-u\|_{\dot B_{2,1}^{\frac{d}{2}-1}}\|u\|_{\dot B_{2,1}^{\frac{d}{2}+1}}+\|w-u\|_{\dot B_{2,1}^{\frac{d}{2}-1}}\|w\|_{\dot B_{2,1}^{\frac{d}{2}}}+\|w\|_{\dot B_{2,1}^{\frac{d}{2}-1}}\|w\|_{\dot B_{2,1}^{\frac{d}{2}+1}}. \end{aligned}\end{equation}
The first and third terms of the right-hand side have already been dealt with in the previous paragraph on high frequencies. We have: $$\displaylines{\int_0^t (1+c_0 t)^{\frac{d+1}{2}}e^{-(t-\tau)}\left(\|w-u\|_{\dot B_{2,1}^{\frac{d}{2}-1}}\|u\|_{\dot B_{2,1}^{\frac{d}{2}+1}}+\|w\|_{\dot B_{2,1}^{\frac{d}{2}-1}}\|w\|_{\dot B_{2,1}^{\frac{d}{2}+1}}\right)d\tau \hfill \cr \lesssim(\mathcal{Z}_0)^2+\mathcal{Z}_0 \ \underset{t\geq 0}{\sup}\left((1+c_0 t)^{\frac{d+1}{2}}\|w-u\|_{\dot B_{2,1}^{\frac{d}{2}-1}}\right).\hfill}$$

For the term $\|w-u\|_{\dot B_{2,1}^{\frac{d}{2}-1}}\|w\|_{\dot B_{2,1}^{\frac{d}{2}}}$, we proceed as follows:
$$\displaylines{\int_0^t (1+c_0 t)^{\frac{d+1}{2}}e^{-(t-\tau)}\|w-u\|_{\dot B_{2,1}^{\frac{d}{2}-1}}\|w\|_{\dot B_{2,1}^{\frac{d}{2}}} d\tau \hfill\cr\lesssim \underset{t\geq 0}{\sup}\left(\|w(t)\|_{\dot B_{2,1}^{\frac{d}{2}}}\right) \ \underset{t\geq 0}{\sup}\left((1+c_0 t)^{\frac{d+1}{2}}\|(w-u)(t)\|_{\dot B_{2,1}^{\frac{d}{2}-1}}\right) \hfill\cr\hfill \times\int_0^t \left(\frac{1+c_0 t}{1+c_0 \tau}\right)^{\frac{d+1}{2}}e^{-(t-\tau)} d\tau\hfill\cr \lesssim\mathcal{Z}_0 \ \underset{t\geq 0}{\sup}\left((1+c_0 t)^{\frac{d+1}{2}}\|(w-u)(t)\|_{\dot B_{2,1}^{\frac{d}{2}-1}}\right). \hfill }$$

Finally, we have: 
$$\displaylines{\int_0^t (1+c_0 t)^{\frac{d+1}{2}}\|\mathbb{P}(u\cdot\nabla)u-(w\cdot\nabla)w\|_{\dot B_{2,1}^{\frac{d}{2}-1}}^l d\tau \hfill\cr \lesssim (\mathcal{Z}_0)^2+\mathcal{Z}_0 \ \underset{t\geq 0}{\sup}\left((1+c_0 t)^{\frac{d+1}{2}}\|(w-u)(t)\|_{\dot B_{2,1}^{\frac{d}{2}-1}}\right)}$$

\item[$\bullet$] Using the following relation $$(u\cdot\nabla)u=\left((u-w)\cdot\nabla\right)u+(w\cdot\nabla)(u-w)+(w\cdot\nabla) w,$$
we have: $$\displaylines{\int_0^t (1+c_0 t)^{\frac{d+1}{2}}e^{-c(t-\tau)}\|\mathbb{P}(u\cdot\nabla)u\|_{\dot B_{2,1}^{\frac{d}{2}-1}}^l d\tau \hfill\cr \lesssim \int_0^t (1+c_0 t)^{\frac{d+1}{2}}e^{-c(t-\tau)}\|\left((u-w)\cdot\nabla\right)u\|_{\dot B_{2,1}^{\frac{d}{2}-1}}^l d\tau \hfill\cr\hfill+\int_0^t (1+c_0 t)^{\frac{d+1}{2}}e^{-c(t-\tau)}\|\left(w\cdot\nabla\right)(u-w)\|_{\dot B_{2,1}^{\frac{d}{2}-1}}^l d\tau \hfill\cr\hfill+\int_0^t (1+c_0 t)^{\frac{d+1}{2}}e^{-c(t-\tau)}\|\left(w\cdot\nabla\right)w\|_{\dot B_{2,1}^{\frac{d}{2}-1}}^l d\tau.}$$
All these terms have already been studied previously and we can deduce the following: $$\displaylines{\int_0^t (1+c_0 t)^{\frac{d+1}{2}}e^{-c(t-\tau)}\|\mathbb{P}(u\cdot\nabla)u\|_{\dot B_{2,1}^{\frac{d}{2}-1}}^l d\tau \hfill\cr \lesssim (\mathcal{Z}_0)^2+\mathcal{Z}_0 \ \underset{t\geq 0}{\sup}\left((1+c_0 t)^{\frac{d+1}{2}} \|w-u\|_{\dot B_{2,1}^{\frac{d}{2}-1}}\right).}$$

By smallness of $\mathcal{Z}_0$, we deduce \eqref{estimée décroissance u-v bf 3d}.
\end{itemize}
\end{proof}
By combining the previous lemma with Lemma \ref{estimée décroissance u,v hf}, we obtain \eqref{estimée décroissance besov u-v 3d}.

\subsubsection{Case of dimension 2}

It remains to deal with the case of dimension 2 with the decay estimate \eqref{estimée décroissance besov u-v 2d} to be proved.

\begin{lemma}\label{estimée décroissance u,v hf 2d}
    We have the following inequality for $d=2$, for all $t\in\R^+$ :
    $$\displaylines{\underset{t\geq 0}{\sup}\left((1+c_0 t)\|(w,u)(t)\|_{\dot B_{2,1}^{0}}^h\right) \hfill\cr \lesssim \mathcal{Z}(0)+\mathcal{Z}(0) \ \underset{t\geq 0}{\sup}\left((1+c_0 t)\|(w-u)(t)\|_{\dot B_{2,1}^{0}}^l\right).}$$
\end{lemma}
\begin{proof}
    To prove this result, we need to repeat the proof of the lemma \ref{estimée décroissance u,v hf}. The only inequality that does not work in the case of dimension 2 is the inequality \eqref{cas 1} where we used $(1+c_0t)\leq(1+c_0 t)^{\frac{d-1}{2}}$ which is true only if $d\geq 3$. We therefore need to reproduce the proof of this lemma and look at $(1+c_0t)\|(w,u)(t)\|_{\dot B_{2,1}^{0}}^h$ instead of $(1+c_0t)^{\frac{3}{2}}\|(w,u)(t)\|_{\dot B_{2,1}^{0}}^h$.
\end{proof}
 
\begin{lemma}
    We have the following estimate at low frequencies for $d=2$ : 
    \begin{equation}\underset{t\geq 0}{\sup}\left((1+c_0 t)\|w-u\|_{\dot B_{2,1}^{0}}^l\right)\lesssim \mathcal{Z}_0+ \mathcal{Z}_0 \ \underset{t\geq 0}{\sup}\left(1+c_0 t)\|w-u\|_{\dot B_{2,1}^{0}}^h\right).\end{equation}
\end{lemma}
\begin{proof}
    Compared to the proof of \eqref{estimée décroissance u-v bf 3d}, we have to bound $(1+c_0t)\|(w-u)(t)\|_{\dot B_{2,1}^{\frac{d}{2}+1}}^l$ instead of $(1+c_0t)^{\frac{d-1}{2}}\|(w-u)(t)\|_{\dot B_{2,1}^{\frac{d}{2}-1}}^l$. In fact, \eqref{cas 2} does not hold for $d=2$. To remedy this, we need to use: $$\|v\|_{\dot B_{2,1}^{0}}^l= \sum_{j\in\Z^{-}}2^{-j} \|\dot\Delta_j v\|_{L^2}2^{j}\lesssim \underset{j\in\Z^{-}}{\sup} \left(2^{-j} \|\dot\Delta_j v\|_{L^2}\right) = \|v\|_{\dot B_{2,\infty}^{-1}}.$$

    Then, by lemma \ref{Produit espace de Besov}, we have:
    \begin{align*}  \|(w\cdot\nabla) (u-w)\|_{\dot B_{2,1}^{0}}^l & \lesssim  \|(w\cdot\nabla) (u-w)\|_{\dot B_{2,\infty}^{-1}}^l \\ & \lesssim \|w\|_{\dot B_{2,1}^1}\|w-u\|_{\dot B_{2,1}^{0}},\end{align*} and we can then proceed as in Lemma \ref{lemme dimension 3}.
\end{proof}

\subsection{Estimates assuming more regularities}
Let us begin with a proposition for the propagation of regularity:
\begin{prop}
Assume the same assumptions as in Theorem \ref{théorème existence et unicité}. If the initial data $(\rho_0,w_0,u_0)$ also satisfies \eqref{condition régularité supplémentaire} for a fixed $k\geq 1$, then there exists a unique solution $(\rho,w,u)$ satisfying the following a priori estimate: 
\begin{equation}\label{inégalité fonctionnelle k}
    \mathcal{L}_k(t)+\int_0^t\mathcal{H}_k(\tau) d\tau \leq \mathcal{L}_k(0),
\end{equation}
where
\begin{equation}\label{définition L_k}\mathcal{L}_k(t)\mathrel{\mathop:}=\mathcal{L}(t)+\frac{1}{2c_B^2}\|w(t)\|_{\dot B_{2,1}^{\frac{d}{2}+1+k}}+\|u(t)\|_{\dot B_{2,1}^{\frac{d}{2}-1+k}}, \end{equation}

\begin{equation}\label{définition H_k}\mathcal{H}_k(t)\mathrel{\mathop:}=\mathcal{H}(t)+\frac{1}{2c_B^2}\|w(t)\|_{\dot B_{2,1}^{\frac{d}{2}+1+k}}+\frac{1}{2c_B}\|u(t)\|_{\dot B_{2,1}^{\frac{d}{2}+1+k}},\end{equation}
    with $\mathcal{L}$ and $\mathcal{H}$ defined in \eqref{définition de L} and \eqref{définition de H}.
\end{prop}

\begin{proof}
   Arguing as in the proof of Lemma \ref{lemme estimée sur v}, we get: $$\displaylines{\|u(t)\|_{\dot B_{2,1}^{\frac{d}{2}-1+k}}+\frac{1}{c_B}\int_0^t \|u\|_{\dot B_{2,1}^{\frac{d}{2}+1+k}}d\tau\leq \|u_0\|_{\dot B_{2,1}^{\frac{d}{2}-1+k}}+\int_0^t \|(u\cdot\nabla)u\|_{\dot B_{2,1}^{\frac{d}{2}-1+k}} d\tau \hfill\cr\hfill+\int_0^t \|\rho(w-u)\|_{\dot B_{2,1}^{\frac{d}{2}-1+k}} d\tau.}$$

By Lemma \ref{Produit espace de Besov} and by interpolation, we have: 
\begin{align*}
    \|(u\cdot\nabla)u\|_{\dot B_{2,1}^{\frac{d}{2}-1+k}} & \lesssim \|u\|_{L^\infty}\|\nabla u\|_{\dot B_{2,1}^{\frac{d}{2}-1+k}}+\|u\|_{\dot B_{2,1}^{\frac{d}{2}-1+k}}\|\nabla u\|_{L^\infty} \\
    &  \lesssim \|u\|_{\dot B_{2,1}^{\frac{d}{2}}}\|\nabla u\|_{\dot B_{2,1}^{\frac{d}{2}-1+k}}+\|u\|_{\dot B_{2,1}^{\frac{d}{2}-1+k}}\|u\|_{\dot B_{2,1}^{\frac{d}{2}+1}} \\
    & \lesssim\sqrt{\|u\|_{\dot B_{2,1}^{\frac{d}{2}-1}}\|u\|_{\dot B_{2,1}^{\frac{d}{2}+1}}\|u\|_{\dot B_{2,1}^{\frac{d}{2}-1+k}}\|u\|_{\dot B_{2,1}^{\frac{d}{2}+1+k}} } \\ & ~~~~ +\|u\|_{\dot B_{2,1}^{\frac{d}{2}-1+k}}\|u\|_{\dot B_{2,1}^{\frac{d}{2}+1}} \\
    & \lesssim \|u\|_{\dot B_{2,1}^{\frac{d}{2}-1}}\|u\|_{\dot B_{2,1}^{\frac{d}{2}+1+k}}+\|u\|_{\dot B_{2,1}^{\frac{d}{2}-1+k}}\|u\|_{\dot B_{2,1}^{\frac{d}{2}+1}}.
\end{align*}

We have also: 
\begin{align*}
    \|\rho(w-u)\|_{\dot B_{2,1}^{\frac{d}{2}-1+k}} & \lesssim \|\rho\|_{L^\infty} \|w-u\|_{\dot B_{2,1}^{\frac{d}{2}-1+k}}+\|\rho\|_{\dot B_{2,1}^{\frac{d}{2}-1+k}}\|w-u\|_{L^\infty} \\
    & \lesssim \|\rho\|_{\dot B_{2,1}^{\frac{d}{2}}}(\|w-u\|_{\dot B_{2,1}^{\frac{d}{2}-1}}^l+\|(w,u)\|_{\dot B_{2,1}^{\frac{d}{2}+1+k}}^h) \\ & + (\|\rho\|_{\dot B_{2,1}^h}^l+\|\rho\|_{\dot B_{2,1}^{\frac{d}{2}+k}}^h)(\|w-u\|_{\dot B_{2,1}^{\frac{d}{2}-1}}^l+\|(w,u)\|_{\dot B_{2,1}^{\frac{d}{2}+1}}^h).
\end{align*}

We then have: 
\begin{multline}\label{estimée u kieme}
        \|u(t)\|_{\dot B_{2,1}^{\frac{d}{2}-1+k}}+\frac{1}{c_B}\int_0^t \|u\|_{\dot B_{2,1}^{\frac{d}{2}+1+k}}d\tau  \\ \leq \|u_0\|_{\dot B_{2,1}^{\frac{d}{2}-1+k}}  +C\int_0^t \|u\|_{\dot B_{2,1}^{\frac{d}{2}-1}}\|u\|_{\dot B_{2,1}^{\frac{d}{2}+1+k}} d\tau  +C\int_0^t \|u\|_{\dot B_{2,1}^{\frac{d}{2}-1+k}}\|u\|_{\dot B_{2,1}^{\frac{d}{2}+1}} d\tau \\ +C\int_0^t \|\rho\|_{\dot B_{2,1}^\frac{d}{2}}(\|w-u\|_{\dot B_{2,1}^{\frac{d}{2}-1}}^l+\|w,u\|_{\dot B_{2,1}^{\frac{d}{2}+1+k}}^h)d\tau \\ +C\int_0^t \|\rho\|_{\dot B_{2,1}^{\frac{d}{2}+k}}(\|w-u\|_{\dot B_{2,1}^{\frac{d}{2}-1}}^l+\|w,u\|_{\dot B_{2,1}^{\frac{d}{2}+1}}^h) d\tau.
\end{multline}

For the second equation of \eqref{Euler-Navier-Stokes2}, viewed in $\dot B_{2,1}^{\frac{d}{2}+1+k}$, we have, instead of \eqref{estimée sur u} : 
\begin{multline}\label{estimée w kieme}
    \|w(t)\|_{\dot B_{2,1}^{\frac{d}{2}+1+k}}+\int_0^t \|w\|_{\dot B_{2,1}^{\frac{d}{2}+1+k}} d\tau \\ \leq \|w_0\|_{\dot B_{2,1}^{\frac{d}{2}+1+k}}+c_B \int_0^t \|u\|_{\dot B_{2,1}^{\frac{d}{2}+1+k}} d\tau +C\int_0^t \|w\|_{\dot B_{2,1}^{\frac{d}{2}+1}}\|w\|_{\dot B_{2,1}^{\frac{d}{2}+1+k}} d\tau.
    \end{multline}

In a similar way to obtaining \eqref{estimée sur rho}, we have: 
$$\displaylines{\|\rho(t)\|_{\dot B_{2,1}^{\frac{d}{2}+k}}\leq \|\rho_0\|_{\dot B_{2,1}^{\frac{d}{2}+k}}+C\int_0^t \|\rho\|_{\dot B_{2,1}^{\frac{d}{2}}}\|w\|_{\dot B_{2,1}^{\frac{d}{2}+1+k}}d\tau \hfill\cr\hfill +C\int_0^t \|\rho\|_{\dot B_{2,1}^{\frac{d}{2}+k}}\|u\|_{\dot B_{2,1}^{\frac{d}{2}+1}} d\tau.}$$

By the smallness of $\int_0^T \|u\|_{\dot B_{2,1}^{\frac{d}{2}+1}} d\tau$ and by assuming $\int_0^T \|w\|_{\dot B_{2,1}^{\frac{d}{2}+1+k}} d\tau<+\infty$, we have \begin{equation}\label{estimée rho kieme}\|\rho\|_{L_T^\infty(\dot B_{2,1}^{\frac{d}{2}})}<+\infty.\end{equation}

By combining estimates \eqref{estimée u kieme}, \eqref{estimée w kieme} and \eqref{estimée rho kieme}, as well as a priori estimates \eqref{estimée théorème}, we deduce \eqref{inégalité fonctionnelle k}.
\end{proof}

As for obtaining \eqref{taux de décroissance L2}, proving \eqref{taux de décroissance L2 k} is based Nash's method again.

By interpolation, we have: 
$$\|v\|_{\dot B_{2,1}^{\frac{d}{2}-1+k}}^l \lesssim \left(\|v\|_{\dot B_{2,1}^{\frac{d}{2}+1+k}}^l\right)^{1-\theta_{k}} \|v\|_{\dot B_{2,\infty}^{-\frac{d}{2}}}^{\theta_{k}}, \quad \text{where} \ \theta_{k}\mathrel{\mathop:}=\frac{2}{d+k+1}.$$
We deduce: 
$$\mathcal{H}_{k}(t)\geq c_0 \mathcal{L}_k^{1-\theta_{k}}(t),$$

hence $$\frac{d}{dt}\mathcal{L}_k+c_0 \mathcal{L}_k^{\frac{1}{1-\theta_k}}\leq 0.$$

Thus we have \begin{equation}\label{estimée décroissance kieme besov}\mathcal{L}_k(t)\leq (1+c_0 t)^{1-\frac{1}{\theta_k}}\mathcal{L}_k(0)=(1+c_0 t)^{-\frac{d+k-1}{2}}\mathcal{L}_k(0).\end{equation}

By injection $\dot B_{2,1}^{\frac{d}{2}}(\R^d)\hookrightarrow L^\infty(\R^d)$, we deduce:

$$\|\nabla^{k-1} v(t)\|_{L^\infty}\lesssim \|v(t)\|_{\dot B_{2,1}^{\frac{d}{2}-1+k}}\leq C (1+c_0 t)^{-\frac{d+k-1}{2}}.$$

We also have by interpolation:
$$\|\nabla^{k-1} v\|_{L^2}\simeq \|v\|_{\dot B_{2,2}^{k-1}}\lesssim \|v\|_{\dot B_{2,\infty}^{-\frac{d}{2}}}^{1-\alpha_k}\|v\|_{\dot B_{2,1}^{\frac{d}{2}+k-1}}^{\alpha_k}, \quad \text{where} \ \alpha_k\mathrel{\mathop:}=\frac{\frac{d}{2}-1+k}{d-1+k}.$$

We then have by \eqref{estimée décroissance kieme besov}: 
$$\|\nabla^{k-1}v\|_{L^2}\leq C (1+c_0 t)^{-\left(\frac{k-1}{2}+\frac{d}{4}\right)}.$$

We can therefore deduce \eqref{taux de décroissance L2 k}.

\section*{Appendix}

\subsection{Study of a matrix}
We have the following lemma concerning the study of the exponential of the matrix associated with the linearised system \eqref{Euler-Navier-Stokes2}:
\begin{lemma}\label{semi-groupe}
For all $t>0$ and $|\xi|<1$, we have the following equality:
$$\exp\left(-t\begin{pmatrix}
    1 & -|\xi|^2 \\ 0 & |\xi|^2
\end{pmatrix}\right)=\begin{pmatrix}
    e^{-t} & g(t,|\xi|) \\ 0 & e^{-t|\xi|^2}
\end{pmatrix},$$ where $g(t,\xi)\mathrel{\mathop:}=-\frac{e^{-t}-e^{-t|\xi|^2}}{1-|\xi|^2}|\xi|^2$ satisfies the inequality $$|g(t,|\xi|)|\leq 2t e^{-t|\xi|^2}.$$
\end{lemma}

\subsection{Useful Lemmas and Estimates}
Here we recall some classical lemmas on differential equations, integration, and some basic properties on Besov spaces and product estimates that have been used repeatedly in the article.
\begin{lemma}\label{lemme edo}
Let $\displaystyle X:[0,T]\rightarrow \R_+$ be a continuous function such that $X^2$ is differentiable. Suppose there is a constant $c\geq 0$ and a measurable function $A:[0,T]\rightarrow \R_+$ such that $$\frac{1}{2}\frac{d}{dt}X^2+c X^2\leq A X \quad a.e.\ \text{on} \ [0,T].$$
Then, for all $t\in [0,T]$, we have: $$X(t)+c\int_0^t X(\tau)\,d\tau\leq X_0 +\int_0^t A(\tau) \,d\tau.$$
\end{lemma}

The following result is classic: see for example \cite{RD}.
\begin{lemma}\label{lemme edo2}
    Let $T>0$. Let $\mathcal{L}:[0,T]\to \R$ and $H:[0,T]\to \R$ be two continuous non-negative functions on $[0,T]$ such that $\mathcal{L}(0)< \alpha$ with $\alpha$ a sufficiently small non-negative real number and  $$\mathcal{L}(t)+c\int_0^t \mathcal{H}(\tau)d\tau \leq \mathcal{L}_0+C\int_0^t \mathcal{L}(\tau)\mathcal{H}(\tau)d\tau.$$ Then, for all $t\in[0,T]$, we have: $$\mathcal{L}(t)+\frac{c}{2}\int_0^t \mathcal{H}(\tau)d\tau \leq \mathcal{L}(0).$$
\end{lemma}
\begin{proof}
    Let $\displaystyle\alpha\in ]0,\frac{c}{2C}[$. We set $T_0=\underset{T_1\in[0,T]}{\sup}\{\underset{t\in[0,T_1]}{\sup}\mathcal{L}(t)\leq \alpha \}$. This supremum exists since the above set is non-empty ($0$ belongs to this set). Since $\mathcal{L}$ is continuous, we have $T_0>0$. Consequently $$\displaylines{\mathcal{L}(T_0)+c\int_0^{T_0}\mathcal{H}(\tau)d\tau \leq \mathcal{L}(0)+C\int_0^{T_0}\mathcal{H}(\tau)\mathcal{L}(\tau)d\tau\leq \mathcal{L}(0)+\alpha C\int_0^{T_0}\mathcal{H}(\tau)d\tau \hfill\cr\hfill \leq \mathcal{L}(0)+\frac{c}{2}\int_0^{T_0}\mathcal{H}(\tau)d\tau.}$$

    We have also: $$\mathcal{L}(T_0)+\frac{c}{2}\int_0^{T_0} \mathcal{H}(\tau)d\tau \leq \mathcal{L}(0).$$

    As $\mathcal{L}(t)\leq \mathcal{L}(T_0)$ for all $t\in[0,T_0]$ and $\displaystyle \int_0^t \mathcal{H}(\tau) d\tau \leq \int_0^{T_0} \mathcal{H}(\tau) d\tau$, we obtain by the previous inequality: $$\mathcal{L}(t)< \alpha \quad \forall t\in[0,T_0].$$

    By continuity of $\mathcal{L}$, we must have $T_0=T$, so the inequality in the statement holds for all $t\in[0,T]$.
\end{proof}
We leave it to the reader to check the following lemma.
\begin{lemma}\label{int exp}    Let $\beta>0$ and $c_0>0$. Then there exists $C$ depending only on $\beta>0$ and $c_0>0$ such that for all $t\geq 0$: $$\int_0^t \left(\frac{1+c_0 t}{1+c_0 \tau}\right)^\beta e^{-c(t-\tau)}d\tau\leq C.$$
\end{lemma}

The following lemmas are classical results, see for example \cite{BCD}.

\begin{lemma} \label{Produit espace de Besov}
For $d\geq 1$, $s,s'\leq \frac{d}{2}$ satisfying $s+s'>0$ the numerical product extends into a continuous application from $\dot B_{2,1}^{s}(\R^d)\times \dot B_{2,1}^{s'}(\R^d)$ to $\dot B_{2,1}^{s+s'-\frac{d}{2}}(\R^d)$. 

For the case $s+s'=0$, the numerical product extends into a continuous application from $\dot B_{2,1}^{s}(\R^d)\times \dot B_{2,\infty}^{s'}(\R^d)$ to $\dot B_{2,\infty}^{s+s'-\frac{d}{2}}(\R^d)$.

$ \dot B_{2,1}^{\frac{d}{2}}$ is a multiplicative algebra for $d\geq 1$.

For $d\geq 1$, we have for $(u,v)\in \dot B_{2,1}^{\frac{d}{2}}\cap  \dot B_{2,1}^{\frac{d}{2}+1}$ that $uv\in  \dot B_{2,1}^{\frac{d}{2}+1}$ and the following inequality: $$\|uv\|_{ \dot B_{2,1}^{\frac{d}{2}+1}}\lesssim \|u\|_{ \dot B_{2,1}^{\frac{d}{2}}}\|v\|_{ \dot B_{2,1}^{\frac{d}{2}+1}}+\|u\|_{ \dot B_{2,1}^{\frac{d}{2}+1}}\|v\|_{ \dot B_{2,1}^{\frac{d}{2}}}.$$
\end{lemma}

\begin{remark}
These continuity properties of the product can easily be generalized to the spaces $\tilde{L}_T^q(\dot B_{p,r}^s)$ defined in \eqref{norme tilde}. The general principle is that the time exponent $q$ behaves according to Hölder's inequality.
\end{remark}
\begin{lemma}\label{commutateur}
    Let $v\in \dot B_{2,1}^{\frac{d}{2}-1}$ and $f\in\dot B_{2,1}^s$ with $s\in ]-\frac{d}{2},\frac{d}{2}+1[$. There is a constant $C$ depending on $s$ and $d$, and a sequence $(c_j)_{j\in\Z}$ satisfying $\sum_{j\in\Z}c_j=1$ such that 
    $$\|[\dot\Delta_j,v\cdot\nabla]f\|_{L^2}\leq C c_j 2^{-js}\|\nabla v\|_{\dot B_{2,1}^{\frac{d}{2}}}\|f\|_{\dot B_{2,1}^s}.$$
\end{lemma}


\begin{thebibliography}{99}
 \bibitem{BCD} 
H. Bahouri, J.-Y. Chemin and  R. Danchin: {\it Fourier Analysis and Nonlinear Partial Differential Equations,} Grundlehren der mathematischen Wissenschaften, {\bf 343}, 
Springer, 2011.

\bibitem{Baranger} C. Baranger, L. Boudin, P.-E. Jabin, S. Mancini: {\it A model of biospray for
the upper airways},  CEMRACS 2004-mathematics and applications to biology and medicine, ESAIM Proc., 14, (2005), 41–47.

\bibitem{Choi2} Y.-P. Choi and J. Jung, {\it On the Cauchy problem for the pressureless~Euler–Navier–Stokes system in the whole space}, J. Math. Fluid Mech., 23, (2021), Article no. 99.

\bibitem{Choi} Y.-P. Choi, J. Jung and J. Kim: {\it A revisit to the pressureless Euler--Navier-Stokes system in the whole space and its optimal temporal decay}, J. Differential Equations {\bf 401} (2024), 231--281.

\bibitem{RD} R. Danchin, {\it Partially dissipative systems in the critical regularity setting, and strong relaxation limit}, EMS Surv. Math. Sci. {\bf 9} (2022), no.~1, 135--192.

\bibitem{LucasToulouse} L. Ertzbischoff, {\it Global derivation of a Boussinesq-Navier-Stokes type system from fluid-kinetic equations}, Ann. Fac. Sci. Toulouse Math. (6) {\bf 33} (2024), no.~4, 1059--1154.

\bibitem{Michel et Daniel} D. Han-Kwan and D. Michel: {\it On hydrodynamic limits of the Vlasov-Navier-Stokes system}, Mem. Amer. Math. Soc. {\bf 302} (2024), no.~1516, v+115 pp.

\bibitem{Danchin2} R. Danchin: {\it An elementary approach to the pressureless Euler-Navier-Stokes system}, prepublication arXiv:2602.06821.

\bibitem{Danchin} R. Danchin: {\it Fujita-Kato solutions and optimal time decay for the Vlasov-Navier-Stokes system in the whole space}, Arch. Ration. Mech. Anal. {\bf 250} (2026), no.~2, Paper No. 13, 45 pp.

\bibitem{Huang} F. Huang, H.Tang, G.Wu and W.Zou: {\it Global well-posedness and large-time behavior of classical solutions to the Euler-Navier-Stokes system in $\R^3$}, J. Differential Equations {\bf 410} (2024), 76--112.

\bibitem{O Rourke}  P. J. O’Rourke: {\it Collective drop effects on vaporizing liquid sprays}, Phd Thesis Princeton University, 1981.

\bibitem{Nash} J. Nash: {\it Continuity of solutions of parabolic and elliptic equations.} Amer. J. Math., {\bf 80}, (1958), 931–954.

\end{thebibliography}
\end{document}